\documentclass[12pt]{amsart}
\usepackage{amsthm}
\usepackage{aliascnt}
\usepackage{hyperref}
\usepackage[capitalize,nameinlink]{cleveref}

\usepackage{amssymb}
\usepackage{epsfig}
\usepackage{amsfonts}
\usepackage{amscd,color}
\usepackage{tikz}
\usetikzlibrary{cd}
\usetikzlibrary{tqft}
\usepackage{youngtab}
\usepackage{ytableau}
\usepackage[notref, notcite,final]{showkeys}
\usepackage{comment}

\usepackage[all]{xy}
\usepackage{xypic,amscd,epsfig}
\usepackage{relsize}

\newcommand{\Mb}{\overline{\mathcal{M}}}

\newcommand{\Z}{\mathbb{Z}}

\newcommand{\C}{\mathbb{C}}

\newcommand{\bP}{\mathbb{P}}
\newcommand{\cO}{\mathcal{O}}
\newcommand{\cS}{\mathcal{S}}
\newcommand{\cQ}{\mathcal{Q}}

\newcommand{\GL}{\,\mathrm{GL}}
\newcommand{\QH}{\,\mathrm{QH}}
\newcommand{\QK}{\,\mathrm{QK}}
\newcommand{\K}{\,\mathrm{K}}

\newcommand{\Gr}{\,\mathrm{Gr}}

\newcommand{\Kpt}{\mathrm{Rep}_{T}}
\newcommand{\T}{\,T}

\newcommand{\mb}[1]{\mathbb{#1}}
\newcommand{\mc}[1]{\mathcal{#1}}

\newcommand{\bs}[1]{\boldsymbol{#1}}
\newcommand{\op}[1]{\operatorname{#1}}

\newcommand{\ev}{\mathrm{ev}}

\newcommand{\pt}{\operatorname{pt}}
\newcommand{\ch}{\operatorname{ch}}

\newcommand{\Eu}{\operatorname{Eu}}
\newcommand{\YB}{\operatorname{YB}}
\newcommand{\Rep}{\operatorname{Rep}}

\newcommand{\bbV}{\bold{V}}
\newcommand{\ve}{\varepsilon}

\newcommand{\bfe}{\mathbf{e}}

\newcommand{\frob}[1]{\left(\!\left(#1\right)\!\right)}

\newtheorem{defn}{Definition}[section]

\newtheorem{remark}{Remark}[section]
\newtheorem{thm}{Theorem}[section]
\newtheorem{prop}{Proposition}[section]

\newaliascnt{cor}{prop}
\newtheorem{cor}[cor]{Corollary}
\aliascntresetthe{cor}

\newaliascnt{lemma}{prop}
\newtheorem{lemma}[lemma]{Lemma}
\aliascntresetthe{lemma}

\newaliascnt{example}{prop}
\newtheorem{example}[example]{Example}
\aliascntresetthe{example}

\newaliascnt{conjecture}{prop}

\aliascntresetthe{conjecture}

\crefname{prop}{proposition}{propositions}
\Crefname{prop}{Proposition}{Propositions}

\crefname{lemma}{lemma}{lemmas}
\Crefname{lemma}{Lemma}{Lemmas}

\crefname{cor}{corollary}{corollaries}
\Crefname{cor}{Corollary}{Corollaries}

\crefname{example}{example}{examples}
\Crefname{example}{Example}{Examples}

\crefname{conjecture}{conjecture}{conjectures}
\Crefname{conjecture}{Conjecture}{Conjectures}

\numberwithin{equation}{section}

\usepackage{chngcntr}
\counterwithout{figure}{section}
\crefname{figure}{Fig.}{Figs.}
\Crefname{figure}{Fig.}{Figs.}

\newcommand{\connecting}[2]{\tikz[baseline=1ex]{
\draw[black] (0,0) -- (0,0.5);
\node at (0,-0.2) {$#2$};
\node at (0,0.7) {$#2$};
\draw[black] (-0.25, 0.25) -- (-0.05, 0.25);
\draw[black] (0.05, 0.25) -- (0.25, 0.25);
\node at (0.39,0.25) {$#1$};
\node at (-0.39,0.25) {$#1$};}}

\newcommand{\avoiding}[2]{\tikz[baseline=1ex]{
\node at (0.39,0.25) {$#2$};
\node at (-0.39,0.25) {$#1$};
\draw[black] (-0.25, 0.25) .. controls (0, 0.25) .. (0,0);
\draw[black] (0, 0.5) .. controls (0, 0.25) .. (0.25,0.25);
\node at (0,-0.2) {$#1$};
\node at (0,0.7) {$#2$};
} }

\newcommand{\passing}[2]{\tikz[baseline=1ex]{
\draw[black] (-0.25,0.25) -- (0.25,0.25);
\node at (0.39,0.25) {$#1$};
\node at (-0.39,0.25) {$#1$};
\draw[black] (0, 0) -- (0, 0.2);
\draw[black] (0, 0.3) -- (0, 0.5);
\node at (0,-0.2) {$#2$};
\node at (0,0.7) {$#2$};}}

\newcommand{\Connecting}[4]{\tikz[baseline=1ex]{
\draw[black] (0,0) -- (0,0.5);
\node at (0,-0.2) {$#4$};
\node at (0,0.7) {$#2$};
\draw[black] (-0.25, 0.25) -- (-0.05, 0.25);
\draw[black] (0.05, 0.25) -- (0.25, 0.25);
\node at (0.39,0.25) {$#3$};
\node at (-0.39,0.25) {$#1$};}}

\newcommand{\Avoiding}[4]{\tikz[baseline=1ex]{
\node at (0.39,0.25) {$#3$};
\node at (-0.39,0.25) {$#1$};
\draw[black] (-0.25, 0.25) .. controls (0, 0.25) .. (0,0);
\draw[black] (0, 0.5) .. controls (0, 0.25) .. (0.25,0.25);
\node at (0,-0.2) {$#4$};
\node at (0,0.7) {$#2$};
} }

\begin{document}

\title[QK structure constants and cylindric Grothendieck polynomials]{Quantum K-structure constants and cylindric Grothendieck polynomials for Grassmannians}

\author{Vassily Gorbounov}
\address{Faculty of Mathematics, National Research University, Higher School of Economics, Usacheva 6, 119048 Moscow, Russia} 
\email{vgorb10@gmail.com}

\author{Christian Korff}
\address{School of Mathematics and {S}tatistics, University of Glasgow, Glasgow G12 8QQ, UK}
\email{Christian.Korff@glasgow.ac.uk}

\author{Leonardo C.~Mihalcea}
\address{
Department of Mathematics, 
460 McBryde Hall, 225 Stanger St.
Virginia Tech University, 
Blacksburg, VA 24061
USA
}
\email{lmihalce@vt.edu}


\date{21 October 2026}

\begin{abstract} We use the partition functions of a five-vertex exactly solvable lattice model from statistical mechanics to give new combinatorial formulae for the Schubert structure constants of torus-equivariant quantum K-theory rings of Grassmannians. The proofs build on our earlier results where we identified the transfer matrices of these lattice models geometrically as multiplications by the Hirzebruch $\lambda_y$ classes of the tautological subbundle and the (dual) quotient bundle. The structure constants are solutions of a system of linear equations with coefficients given by the partition function. This system may be solved explicitly by specializing 
the Hirzebruch parameter $y$ to an equivariant parameter, leading to 
explicit formulae using Cramer's rule. Furthermore, we give an alternative combinatorial construction of the partition functions in terms of sums over cylindric set-valued tableaux, which we call cylindric Grothendieck polynomials. The latter generalise McNamara's factorial skew Grothendieck and Yeliussizov's modified Grothendieck polynomials and can be thought of as an equivariant K-theoretic version of Postnikov's toric Schur function. We also provide new polynomial representatives for equivariant K-theory classes of Richardson varieties and their variants.
\end{abstract}

\maketitle
\tableofcontents

\section{Introduction} Let $X=\Gr(k,n)$ be the Grassmann manifold of $k$-dimensional 
linear subspaces in $\C^n$, considered with the action of $T \subset \GL_n$, 
the maximal torus of diagonal matrices. Among the most studied objects in 
Schubert calculus are the Schubert structure constants
$c_{\lambda,\mu}^\nu \in \Z_{\ge 0}$ arising from the multiplication 
\begin{equation}\label{E:multH} \sigma_\lambda \cdot \sigma_\mu = \sum c_{\lambda, \mu}^\nu \sigma_\nu \end{equation} 
of two Schubert classes $\sigma_\lambda, \sigma_\mu$ in the cohomology ring of $X$. Here 
$\lambda,\mu,\nu$
are partitions included in the $k \times (n-k)$ rectangle. 
There are a number of ways the structure constants $c_{\lambda,\mu}^\nu$ arise purely combinatorially, and 
the route closest to our approach is 
given by the Schur expansion of skew Schur functions \cite{FultonYoungTableaux}:
\begin{equation}\label{E:skewSchur} s_{\nu/\lambda} = \sum c_{\lambda, \mu}^\nu s_\mu \/. \end{equation}
Algebraically, the skew Schur function arises naturally when considering the cohomology ring as a Frobenius algebra \cite{abrams1996two} and computing the co-product of a Schubert class. The multiplication \cref{E:multH} and Frobenius structure can be deformed in a number of ways, using other cohomology theories 
(e.g., equivariant cohomology, or K-theory), or quantum versions of such. 
The common feature of all these rings is that they admit a Schubert basis over an appropriate base ring, thus there are well-defined Schubert structure constants. 

\subsection{Equivariant quantum K-theory}
The current paper will be concerned with $\QK_T(X)$, the $T$-equivariant quantum K-theory
ring of $X$, defined by Givental and Lee \cite{givental:onwdvv,lee:QK}. This is a deformation of the equivariant K-theory ring $\K_T(X)$, and it admits a basis by 
Schubert classes, defined as fundamental classes of Schubert varieties in $X$. Denote by $\cO_\lambda$ 
its Schubert basis, with conventions spelled out 
in \cref{sec:schub-preliminaries} below. 
Then $\QK_T(X)= \bigoplus_\lambda \K_T(\pt)[\![q]\!] \cO_\lambda$, where 
$q$ is the Novikov variable, and $\K_T(\pt) = \Z[\ve_1^{\pm 1}, \ldots,  \ve_n^{\pm 1}]$
is the $T$-equivariant K-theory of the point, a Laurent polynomial ring in the 
$T$-characters $\ve_1, \ldots, \ve_n$ of $\C^n$. The analogue of \cref{E:multH} is
\begin{equation}\label{E:multQKT} 
\cO_\lambda \star \cO_\mu = \sum_{\nu, d \ge 0} C_{\lambda, \mu}^{\nu,d}q^d \cO_\nu \end{equation} 
with structure constants $C_{\lambda, \mu}^{\nu,d} \in \K_T(\pt)$, and $\star$ the quantum K product, a deformation of the usual intersection product ``$\cdot$".
These structure constants specialize to those in various (quantum) cohomology theories of $X$. 
Most interestingly, they are, or are conjectured to be, appropriately positive. For example, 
in the non-equivariant ring $\QK(X)$, the structure constants are integers, and it was 
recently proved in \cite{BCMP:qkpos} that 
\[ (-1)^{|\lambda|+|\mu|-|\nu|-nd} C_{\lambda,\mu}^{\nu,d} \in \Z_{\ge 0} \/. \]
This statement generalizes the positivity of Gromov-Witten invariants in quantum cohomology (when $|\lambda|+|\mu|-|\nu|-nd=0$), and the alternating positivity in K-theory \cite{Buch2002,Brion2002} (when $d=0$); see \cite{Graham2001,anderson.griffeth.miller:Kpos} for equivariant cases.

Given the prominent role of the skew Schur functions in finding a positive formula for 
$c_{\lambda,\mu}^\nu$ (the classical Littlewood-Richardson rule), it is natural to 
ask whether one can also find an 
analogue of \cref{E:skewSchur} in the quantum K setting. A `precursor' in the case of quantum cohomology is given by Postnikov 
\cite{postnikov:affine} who introduced {\em toric Schur functions} $s_{\nu/d/\lambda}$ associated
to a cylindric shape determined by two ordinary partitions $\lambda,\nu$ and an integer $d$. They are
a family of symmetric functions included in the larger set of cylindric 
Schur functions defined by Gessel and Krattenthaler \cite{gessel1997cylindric}, 
and generalize the ordinary skew Schur functions by arranging skew tableaux on a cylinder. 
Furthermore, the Schur expansion of toric Schur functions is non-negative:
\begin{equation} \label{E:toricSchur}
s_{\nu/d/\lambda} = \sum c_{\lambda, \mu}^{\nu,d} s_\mu \/.
\end{equation}
Here $c_{\lambda,\mu}^{\nu,d}$ is the (ordinary, genus $0$, 
$3$-point) Gromov-Witten invariant, which counts the 
number of rational curves of degree $d$ in $X$ passing 
through general translates of the Schubert varieties associated to $\lambda, \mu$ 
and complement of $\nu$. At the same time, 
$c_{\lambda, \mu}^{\nu,d}$ is a structure constant in $\QH^*(X)$, the quantum cohomology ring of $X$.
Positivity of the structure constants follows from geometric reasons, and, unfortunately, at this time there is no known  
`cylindric Littlewood-Richardson rule'. The only manifestly positive formula arises from the 
`quantum=classical' phenomenon of Buch, Kresch and Tamvakis \cite{buch.kresch.ea:gromov-witten}, who proved 
that $c_{\lambda,\mu}^{\nu,d}$ are equal
to certain Schubert structure constants on a 
two-step flag manifold; combinatorial 
positive formulae for those are known \cite{BKPT2016,Buch2015,KnutsonZinnJustin2026}. 

\subsection{Main goals of this article}
Our goals in this paper are:

(I) Find a combinatorial formula for the quantum K-structure constants $C_{\lambda,\mu}^{\nu,d}$;

(II) Define cylindric (factorial, skew) Grothendieck polynomials giving an analogue of \cref{E:skewSchur}
in quantum K-theory, and find a (set-valued) tableau formula for them. 

The formulae we obtain for the structure constants $C_{\lambda,\mu}^{\nu,d}$ arise from solving a 
certain system of linear equations, 
and are combinatorially explicit 
(cf.~\Cref{thm:alg1} and \Cref{thm:Cramer} below), but not cancellation free. Due to their equivariant nature, 
one may think of them as related to formulae for Schubert structure 
constants using Chevalley recursions 
\cite{Mihalcea2006,BCMP:qkchev}. In fact, 
our formulae seem new already in the classical ring
$\K_T(X)$, the equivariant K-theory of Grassmannians. The cylindric Grothendieck polynomials in (II)  generalize both Postnikov's toric Schur functions and McNamara's factorial (skew) Grothendieck polynomials \cite{mcnamara:factorial}, and also Yeliussizov's modified 
Grothendieck polynomials \cite{yeliussizov2019symmetric}.

\subsection{Main technique: partition functions from exactly solvable lattice models}
The novelty in our approach to answer (I) and (II) is that it 
has roots in the physics of lattice models. A key observation from \cite{korff2014quantum} is that Postnikov's toric Schur 
functions may be realized in two (dual) ways, as partition functions of certain exactly solvable 
five-vertex models from 
statistical mechanics called {\em vicious} and {\em osculating walkers} \cite{fisher1984walks,brak1999bethe,bousquet2006three} by considering a cylindrical square lattice. 
The same five-vertex models describe dimer coverings on the honeycomb 
lattice \cite{Huang1996} and are special degenerations of the celebrated six-vertex model in statistical mechanics whose free energy was first calculated by Lieb \cite{lieb1967residual} and Sutherland \cite{sutherland1967exact} using the Bethe ansatz. Baxter showed that these solutions could be understood in the much broader framework of commuting transfer matrices \cite{Baxter:book} and there has been in recent years much activity by several groups to understand symmetric and other special functions in terms of such exactly solvable lattice models; see e.g. \cite{zinnjustin2009sixvertex,brubaker2011schur,bump:mcnamara:nakasuji,korff2014quantum,brubaker2020vertex,brubaker2021colored,korff21,buciumas2022double,corteel2022vertex,brubaker2023frozen,aggarwal2023free,naprienko2024free,brubaker2025lattice} and references therein for a small sample of works.

In this article we generalize the connection between partition functions of exactly solvable 
lattice models and quantum Schubert Calculus to $\QK_T(X)$, i.e. we compute the partition 
functions of the exactly solvable five-vertex models with K-theoretic weights using the 
results from \cite{gorbounov.korff:unpublished,gorbounov2017quantum,gorbounov.korff.mihalcea:YB}. 
We use those to give two formulae for the structure constants $C_{\lambda,\mu}^{\nu,d}$, 
one of which is in form of a determinant.

\subsection{Statement of results} We now give a more precise account of our results. Fix two partitions
$\lambda, \nu$ included in the $k \times (n-k)$ rectangle. The quantum K ring $\QK_T(X)$ is equipped with a Frobenius pairing $(\!(\cdot, \cdot)\!)$, generalizing the usual intersection pairing giving the holomorphic
Euler characteristic. Denote by $\mathcal{I}^\lambda_q$ the dual classes of the Schubert classes, i.e. 
$(\!( \cO_\lambda, \mathcal{I}^{\nu}_q)\!) = \delta_{\lambda,\nu}$. 
Denote also by $0 \to \cS \to \C^n \to \cQ \to 0$ the tautological sequence bundles on $X$, and let 
$\lambda_y(\cQ^\vee) = \sum y^i \wedge^i \cQ^\vee$, regarded as a class in $\K_T(X)[y]$, and where 
$y$ is a formal variable. Using the Frobenius bilinear form  of $QK_T(Gr(k,n))$ (see e.g. \cite[Section 12]{gorbounov.korff.mihalcea:YB} as well as references therein) we identify the partition function of the lattice model with the following matrix element which we call the {\bf geometric partition function} 
\begin{equation} 
Z_\lambda^\nu(z|\ve;q)=(\!(\cO_\lambda \star \mathcal{I}^\nu_q, \lambda_{z_1}(\cQ^\vee) \star \cdots 
\star \lambda_{z_k}(\cQ^\vee) )\!)\;. 
\end{equation}
This is a polynomial function in $\Z[z;\ve^{-1}; q]$, with $\ve^{-1}=(\ve_1^{-1}, \ldots, \ve_n^{-1})$,  symmetric in variables $z=(z_1, \ldots, z_k)$. A key fact proved in 
\cite{gorbounov.korff.mihalcea:YB} and building on results from \cite{gorbounov2017quantum} is that 
the multiplication operator by $\lambda_y(\cQ^\vee)$ is a transfer matrix operator, i.e. it coincides with the 
action of the quantum trace of the monodromy matrix giving the action of a certain Yang-Baxter algebra.
The lattice model associated to this action gives the more familiar graphical formula for the partition function as sums over weighted configurations from composition of transfer matrices
\[ Z_\lambda^\nu(z|\ve;q)= \sum_\gamma \mathrm{wt}(\gamma)\,, \]
where $\gamma$ varies over the set of configurations of a rectangular lattice with $k$ rows, with the top indexed by the $01$ word of $\lambda$, the bottom by the $01$ word of $\nu$, and which are {\em cylindrical}, i.e. the left lattice edge is identified with the right edge. The weight of a configuration $\gamma$ is the product of the weights
of each vertex in it, calculated from the Boltzmann weights associated to the lattice model. We illustrate this in the picture and example below, and we refer to \Cref{sec:graphical-calculus} for precise conventions and definitions.
\begin{figure}[h!]
\centering
\includegraphics[width=.8\textwidth]{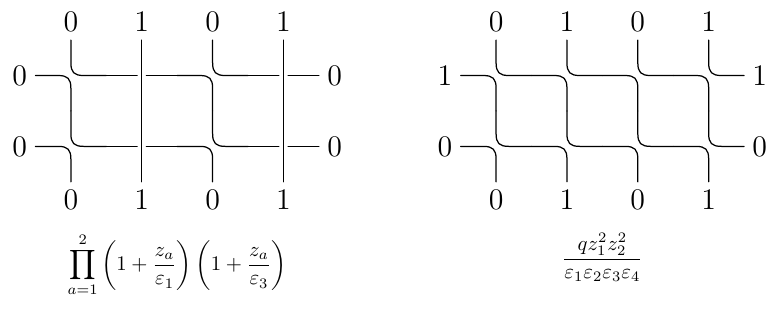} 
\vskip-0.6cm
\caption{The two possible lattice configurations for $\Gr(2,4)$ and $\lambda=\nu=(2,1)$ with $01$ words $J_\lambda=1010$ (right to left reading).}\label{fig:Gr24example-intro}
\end{figure}
\begin{example}\label{ex:Gr24-21}\rm
 Consider $X=\Gr(2,4)$, with $\lambda=\nu=(2,1)$. In this case there are only two possible lattice configurations which are depicted in
 \Cref{fig:Gr24example-intro}. 
    The power of $q$ occurs since we have a line carrying a 1-letter crossing the boundary of the cylinder. The partition function is
\begin{equation}\label{Z24-21-intro}
\begin{aligned}
Z_{(2,1)}^{(2,1)}(z_1,z_2\mid\ve;q)
={}&
\prod_{i=1}^{2}
\left(1+\frac{z_i}{\epsilon_1}\right)
\left(1+\frac{z_i}{\epsilon_3}\right)
+\frac{qz_1^2z_2^2}
{\epsilon_1\epsilon_2\epsilon_3\epsilon_4}.
\end{aligned}
\end{equation}
%\begin{equation}\label{Z24}
%        Z_\lambda^\nu(z_1, \ldots, z_k|\ve;q)=\frac{z_1^2z_2^2}%{\ve_2^2\ve^2_4}+q(1+z_1/\ve_2)(1+z_2/\ve_1)(1+z_1/\ve_4)%(1+z_2/\ve_3)\;.
%\end{equation}
\end{example}
Our first result relates structure constants $C_{\lambda,\mu}^{\nu,d}$ to partition functions:
\begin{thm}[cf. \Cref{prop:structure-constants}]\label{prop:structure-constants-intro} The partition function ${Z}_\lambda^\nu$ satisfies
\begin{equation}\label{E:linsystem-intro} {Z}_\lambda^\nu(z_1, \ldots, z_k|\ve;q) = \sum_{\mu,d} Z_{\emptyset}^{\mu}(z_1, \ldots, z_k|\ve;q) q^d C_{\lambda,\mu}^{\nu,d} \/. \end{equation}
\end{thm}
The main ingredient in the proof is 
that the product $\lambda_{z_1}(\cQ^\vee) \star \cdots \star
 \lambda_{z_k}(\cQ^\vee)$ coincides with the classical product in $\K_T(X)$, thus 
 is independent of $q$. As a corollary we obtain a geometric interpretation of the partition function. 
 
Set $\mc{I}^\lambda=[\cO_{X^\lambda}(-\partial X^\lambda)]$ to be 
the class of the ideal sheaf of the boundary $\partial X^\lambda$ 
of the opposite Schubert variety $X^\lambda$.
\begin{cor}[cf.~\Cref{cor:part-as-class}]\label{cor:part-as-class-intro} Regard the variables $-z_i$ as K-theoretic Chern roots of $\cS$,
i.e., $\prod (1-yz_i) = \lambda_y(\cS)$, and the partition function $Z_\lambda^\nu(z_1, \ldots, z_k|\ve;q)$ 
as an element in $\K_T(X)[q]$. Then
\[ Z_\lambda^\nu(z_1, \ldots, z_k|\ve;q)=\sum_{\mu,d} C_{\lambda,\mu}^{\nu,d}q^d \mathcal{I}^\mu \/.\]
If one specializes to $q=0$, then
\[ Z_\lambda^\nu(z_1, \ldots, z_k|\ve;0) = \cO_\lambda \cdot \mathcal{I}^\nu \/\]
as an element in $\K_T(X)$. 
In particular, 
\[ Z_\lambda^\nu(z_1, \ldots, z_k|\ve;0)= [\cO_{X_\lambda \cap X^\nu}(- X_\lambda \cap \partial X^\nu )] \/, \]
the class associated 
to the boundary $X_\lambda \cap \partial X^\nu$ of the Richardson variety
$X_\lambda \cap X^\nu$.
\end{cor}
Note that the last statement is immediate from Brion's results e.g. in \cite{brion:flagv}. Observe also that $Z_{\emptyset}^{\nu}(z_1, \ldots, z_k|\ve;q)$ is independent of $q$, and in particular it represents the class of the (ordinary) ideal sheaf $\mathcal{I}^\nu$; we omit $q$ and denote it simply by $Z_{\emptyset}^{\nu}(z_1, \ldots, z_k|\ve)$. 

\Cref{prop:structure-constants-intro} may be interpreted as a system of linear equations in $C_{\lambda, \mu}^\nu:=\sum_d C_{\lambda,\mu}^{\nu,d} q^d$, which may be solved using Cramer's rule. A drastic simplification occurs 
in the equivariant case. If one specializes $z$ at a sequence $-\ve_\eta$ coming from the 
weights over the $T$-fixed point $\bfe_\eta$, then $Z_\emptyset^\mu$ is the localization coefficient 
$\langle \mathcal{I}^\mu,\bfe_\eta\rangle$. 
In particular, the system of linear equations \cref{E:linsystem-intro} becomes upper triangular, 
and one obtains an explicit determinant formula for $C_{\lambda, \mu}^\nu$; see \Cref{thm:Cramer} below. 
Using a variant of \Cref{cor:part-as-class-intro}, this solution may also be seen explicitly (not as a determinant):

\begin{thm}[cf. \Cref{thm:alg1}] The Schubert structure constants 
$C_{\lambda, \mu}^\nu=\sum_d C_{\lambda,\mu}^{\nu,d} q^d$ in 
$\QK_T(X)$ are given by the following expression:
\begin{equation}\label{E:alg1-intro}
 C_{\lambda,\mu}^\nu= \sum_\eta \frac{1}{\Eu_\eta} \langle \cO_\mu, \bfe_\eta \rangle
 Z_\lambda^\nu(-\ve_\eta|\ve;q) \/. \end{equation}
 %Here $\Eu_\eta$ is the K-theoretic Euler class at the fixed point $\bfe_\eta$.
 \end{thm}
We remark that $\Eu_\eta$, the K-theoretic Euler class of $X$ at the fixed point $\bfe_\eta$, the localization coefficient 
$\langle \cO_\mu, \bfe_\eta \rangle$, and the specialized partition function, 
are combinatorially explicit (see \Cref{classEu} respectively \Cref{E:locOlambda}). Each of these terms is appropriately positive for a choice of K-theoretic positive roots $1- \ve_i/\ve_j$, but these choices are 
{\em different}, thus the total formula is not positive. Nevertheless, 
this gives a new combinatorial expression for the quantum K structure constants in terms of 
%the graphical calculus of the lattice models which uses 
non-intersecting paths on the cylinder.

We now turn to the cylindric Grothendieck analogue of \cref{E:skewSchur}. 
The special case $\lambda=\emptyset$ of the partition function suggests that 
the polynomial under consideration must be closely related to the (skew) factorial Grothendieck 
polynomials studied by McNamara \cite{mcnamara:factorial}. However, since for $\lambda=\emptyset$ 
the partition function represents classes of {\em ideal sheaves} $\mathcal{I}^\nu$, 
it turns out the most relevant 
polynomials for us are the (single) {\em modified skew Grothendieck polynomials} $G_{\lambda//\mu}$ defined more recently by Yeliussizov \cite{yeliussizov2019symmetric}. These are defined using Buch's stable Grothendieck polynomials \cite{buch:groth-classes}, and are specializations of Grassmannian double Grothendieck polynomials defined by Lascoux and Sch{\"u}tzenberger \cite{lascoux1982structurehopf}.

We define a double, cylindric, generalization of both McNamara's and Yeliussizov's polynomials, using certain cylindric set-valued tableaux. We then relate these to the partition function
$Z_\lambda^\nu(z|\ve;q)$. We achieve this by establishing an explicit map between Buch's set-valued tableaux \cite{Buch2002} when arranged on a cylinder and five-vertex lattice configurations from \cite{gorbounov2017quantum,gorbounov.korff.mihalcea:YB}.

More precisely, let $\hat{\lambda}[0]$
denote the cylindric shape obtained from $\lambda$ by translating it 
with the $k \times (n-k)$ rectangle, and let $\nu[d]$ denote the cylindric
shape obtained by moving $\nu$ with vector $(d,d)$ in SE direction;
see \Cref{fig:cyl_loop} below and \Cref{sec:cylindric-groth} for precise notation. For $\theta_1,\theta_2$ cylindric loops, we define the {\em factorial modified cylindric skew Grothendieck polynomial} 
as \begin{equation}\label{mcsG-intro}
    G_{\theta_1//\theta_2}(x|t)=\sum_{\substack{\rho\subset\theta_2\\
    \text{cyl. rook strip}}}\beta^{|\theta_2/\rho|}G_{\theta_1/\rho}(x|t),
\end{equation}
where the sum ranges over all cylindric loops $\rho$ such that $\theta_2/\rho$ is a {\em cylindric rook strip}, as defined in \Cref{sec:cylindric-groth}. Here $\beta$ is a homogenizing parameter,
and $G_{\theta_1/\rho}(x|t)$ is a cylindric generalization of McNamara's skew factorial Grothendieck polynomials, defined using a cylindric generalization of Buch's set valued tableaux.
Our main result is:
\begin{thm}[cf. \Cref{prop:Z=mcsG}]\label{thm:part-to-groth} Let $\lambda, \nu$ be two partitions in the $k \times (n-k)$ rectangle. Set $\hat\lambda=(\lambda_1+n-k,\ldots,\lambda_k+n-k)$, then:
\begin{equation}\label{E:expansionZ-intro}
Z^\nu_\lambda(x_1,\ldots,x_k|t;q)=\sum_{d=0}^kq^dG_{\hat\lambda[0]//\nu[d]}(x_1,\ldots,x_k|\tau^kt),
    \end{equation}
    where $x_i=1+z_i$, $t_j=1-\ve^{-1}_j$, and $\tau^k t$ is a $k$-shift of the sequence $t=(t_1, \ldots, t_n)$.
\end{thm}
This theorem is a K-theoretic generalization of Postnikov's result from \cite{postnikov:affine} up to a certain duality: the polynomial $G_{\hat\lambda[0]//\nu[d]}$ specialised to $\beta=0$ and $t_j=0$ yields the toric Schur 
function $s_{\hat\lambda[0]/\nu[d]}$. This has degree $k(n-k)+|\lambda|-|\nu|-nd$ and it expands as $s_{\hat\lambda[0]/\nu[d]}=\sum_{\mu}c_{\lambda\mu}^{\nu,d}s_{\mu^\vee}$ with $\mu^\vee$ denoting the complement of the Young diagram in the $k\times(n-k)$ bounding box. In contrast, Postnikov uses $s_{\nu/d/\lambda}$ which has degree $|\nu|-|\lambda|+dn$ and the expansion in \cref{E:toricSchur}.  That both toric Schur functions lead to the same expansion coefficient is due to \cite[Cor.6.9]{postnikov:affine}; cf. \Cref{cor:toricSchur} and \Cref{rmk:toricSchur}. 

We also note that if one combines \Cref{thm:part-to-groth} with the geometric interpretation of the partition function from 
\Cref{cor:part-as-class-intro} one obtains polynomials representing the product $\mathcal{I}^\nu \cdot \cO_\lambda$. Using M{\"o}bius inversion relating ideal and structure sheaves, 
this gives a polynomial representing the class $\cO^\nu \cdot \cO_\lambda$, which is the same as the
equivariant K-theory class of the Richardson variety $X_\lambda \cap X^\nu$ \cite{brion:flagv}. These polynomial representatives seem new, and generalize the well-known fact that (straight, double) Grothendieck polynomials represent Schubert classes in equivariant K-theory \cite{FultonLascoux1994,buch:groth-classes,Buch2002,lenart2006grothendieck}.
\smallskip

{\em Acknowledgments.} LM and CK were supported in part by a Royal Society International Exchanges grant which allowed us to travel to each other's institutions.
LM was partially supported by NSF grant DMS-2152294, and gratefully acknowledges the support
of Charles Simonyi Endowment, which provided funding for the membership at the Institute
of Advanced Study during the 2024-25 Special Year in ‘Algebraic and Geometric Combinatorics’, where this collaboration was initiated.
\medskip

{\em Tool and computational resource disclosure.} The authors used ChatGPT 5.6 and ChatGPT 6.0 Pro to perform experiments, improve Mathematica code for computing examples and partition functions of the lattice models, audit the paper, and search for literature related to Grothendieck polynomials; Yeliussizov's paper
\cite{yeliussizov2019symmetric} and its relevance to our project were identified this way.

\section{Preliminaries}\label{sec:schub-preliminaries}
In this section we set the basic conventions used in this paper, following those from \cite{gorbounov.korff.mihalcea:YB}.
\subsection{Notation and conventions}\label{sec:conventions} 
We work over $\C$, and by a variety we mean a scheme of finite type which is 
reduced and irreducible. For a fixed natural number $n$ denote by $f_1, \ldots, f_n$ 
the standard basis of $\C^n$. Let $T \simeq (\C^*)^n$ be the $n$-dimensional torus, 
acting on $\C^n$ as usual: 
\[(a_1, \ldots, a_n).(z_1, \ldots, z_n) = (a_1z_1, \ldots, a_n z_n) \/. \]
The $T$-module $\C^n$ has a weight space decomposition given by 
$\C^n = \bigoplus_{i=1}^n \C_{\ve_i}$, where $\C_{\ve_i}$ denotes 
the one dimensional $T$-module with action given by the character 
$\ve_i(z_1, \ldots , z_n)= z_i$. Denote 
the representation ring of $T$ by $\Kpt$, and as customary we identify characters by 
their $1$-dimensional 
modules: $\ve_i = [\C_{\ve_i}]$. Then $\Kpt= \Z[\ve_1^{\pm 1}, \ldots , \ve_n^{\pm 1}]$, 
the Laurent polynomial ring in the indeterminates $\ve_i$.
If $y$ is an indeterminate, the $\lambda_y$ class of 
$\C^n$ is the element in $\Kpt[y]$ defined by 
\[ \lambda_y(\C^n) = (1+y \ve_1)\cdot \ldots \cdot (1+ y\ve_n) \/.\]

Let $W:=S_n$ be the symmetric group in $n$ letters, equipped with 
length function $\ell:W \to \mathbb{N}$. 
Denote by $w_0$ its longest element. For $1 \le k \le n$ we denote by 
$W^k_n$ the set of {\bf Grassmannian permutations}, i.e., permutations such
that $w(1) < \ldots < w(k)$ and $w(k+1) < \ldots < w(n)$. 
These are in bijection with $\Pi^k_n$, the set of partitions 
$\lambda=(\lambda_1, \ldots, \lambda_k)$ included in 
the $k \times (n-k)$ rectangle, i.e.,
$n-k \geq \lambda_1 \geq \cdots \ge \lambda_k \geq 0$.
The weight $|\lambda|$ is $\lambda_1 + \ldots + \lambda_k$, and it is equal to the length of the corresponding permutation $w_\lambda$.
A third set in bijection with the two above is
$\mb{J}_n^k$ 
the set of $01$ words $j_1 j_2 \ldots j_n$
of length $n$ with $k$ $0$'s. We will use the following bijections among the sets $W^k_n, \Pi_n^k$ and $\mb{J}_n^k$.
The permutation $w_\lambda \in W^k_n$ associated to the partition $\lambda \subset k \times(n-k)$ is the unique Grassmannian permutation satisfying
\[ w(i)= \lambda_{k-i+1}+i \/; \quad 1 \le i \le k \/. \]
The $01$ word $J_\lambda = j_1 j_2 \ldots j_n$ is the word with $0$'s in positions 
$\lambda_i+k-i+1$, for $1 \le i \le k$.
Graphically, $J_\lambda$ is obtained by tracing the outline of the Young 
diagram of $\lambda$, starting from the SW corner, and placing $0$'s for 
the vertical 
steps, and $1$'s for the horizontal steps. The permutation $w_\lambda$ 
starts with the $k$ positions given by the $0$'s; see \Cref{fig:Young}.
\begin{figure}[h!]
\centering
\includegraphics[width=.2\textwidth]{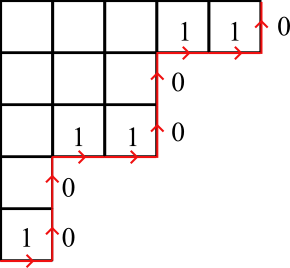} 
\caption{A Young diagram of $\lambda=(5,3,3,1,1)$ and its associated 01-word $J_\lambda=1001100110$ for $\Gr(5;10)$. The corresponding Grassmannian permutation is $w_\lambda=[2,3,6,7,10,1,4,5,8,9]$.}\label{fig:Young}
\end{figure}
Given a partition $\lambda$ we denote by 
\[I_\lambda = \{\lambda_i+k-i+1: 1 \le i \le k\} \] 
the set containing the positions of $0$'s of $J_\lambda$.
We also note that the symmetric group $S_n$ acts on 
the set of Young diagrams in the $k \times (n-k)$ rectangle by permuting the positions of $0$'s and $1$'s in $J_\lambda$.

\begin{example}\label{ex:Ydiagram}\rm With these conventions
$J_\emptyset =00\cdots01\cdots1$ and $J_{(n-k)^k}=11\cdots10\cdots0$ (with $k$ $0$'s). 
The corresponding permutations are: 
\[ w_{\emptyset} = id \/; \quad w_{(n-k)^k} = \begin{pmatrix} 1 & 2 & \ldots & k & k+1 & \ldots & n \\ n- k +1 & n-k+2 & \ldots & n & 1 & \ldots & n-k \end{pmatrix} \/. \]
For another example, let $\lambda = (3,3)$, regarded in the $3 \times (7-3)$ rectangle. Then
\[ J_\lambda = 0 1 1 1 0 0 1 \quad \textrm{ and } \quad w_\lambda = \begin{pmatrix} 1 & 2 & 3 & 4 & 5 & 6 & 7 \\ 1 & 5 & 6 & 2 & 3 & 4 & 7 \end{pmatrix} \/. \]
The set of zero positions is $I(\lambda) = \{ 1, 5,6 \}$. If $s_4=(45)$ is a simple transposition, then $J_{s_4.w_\lambda}= 0110101$, therefore
$s_4.w_{(3,3)} = w_{(3,2,0)}$.\end{example}
%%%%%%%%%%%%%%%%%%%%%%%%%%%%%%%%%%%%%%%%%%%%%%%%%%%%%%%%%%%%%%%%%%%%%
We will use these different combinatorial descriptions interchangeably, noting that the description in terms of partitions and Grassmannian permutations is convenient in the geometric setting while the labelling in terms of binary strings occurs in the discussion of the quantum integrable lattice model and its graphical calculus.

\subsection{Schubert varieties} Let $\Gr(k,n)$ denote the Grassmannian 
parametrizing $k$-dimensional subspaces in $\C^n$. This is a complex projective manifold of 
dimension $k(n-k)$, and it is homogeneous under the action of $\GL_n(\C)$.
Fix the (opposite) standard flag 
\[ F_\bullet: F_1 = \langle f_n \rangle \subset F_2 = \langle f_n, f_{n-1} \rangle \subset \ldots \subset \C^n \/.\]
Let $\lambda=(\lambda_1, \ldots, \lambda_k)$ be a partition included in the $k \times (n-k)$ rectangle. The {\bf Schubert cell}
associated to $\lambda$ and $F_\bullet$ is defined as follows.
Set $p_i=n-k+i-\lambda_i$ and 
\[ X_\lambda^\circ(F_\bullet) = \{ V \in \Gr(k,n): \dim V \cap F_{j} = \# \{i: p_i \le j \}, \textrm{ for } 0 \le j \le n \/. \} \/. \]
This is isomorphic to the affine space $\C^{k(n-k) - |\lambda|}$. The corresponding
{\bf Schubert variety}, denoted by $X_\lambda:=X_\lambda(F_\bullet)$ 
is the (Zariski) closure 
of $X_\lambda^\circ(F_\bullet)$. Similarly, using the opposite
flag $F_\bullet^{opp}$ given by $F_i^{opp} = \langle f_1, \ldots , f_i\rangle$ one defines the opposite Schubert cell 
$X^{\lambda, \circ} \simeq \C^{|\lambda|}$. 
Both types of Schubert cells are $\T$-stable and 
contain the unique $\T$-fixed point
\[\bfe_\lambda= \langle f_{k+1-i+\lambda_i}: i=1 \ldots k \rangle \/. \] 
Finally, both types of Schubert cells give a stratification of the Grassmannian: 
\[ 
\Gr(k,n) = \bigsqcup_\lambda X_\lambda^\circ(F_\bullet)=\bigsqcup_\lambda X^{\lambda,\circ}(F_\bullet^{opp}) \/.
\]
Note that $w_0.X_\lambda= X^{\lambda^\vee}$ where $\lambda^\vee$ is the complement of $\lambda$ in the $k \times (n-k)$ rectangle.
Inclusion of Schubert varieties is given by partition inclusion:
\[ \lambda \subset \mu \quad \Leftrightarrow \quad X_\mu \subset X_\lambda 
\quad \Leftrightarrow \quad X^{\lambda} \subset X^\mu \/. \]

%%%%%%%%%%%%%%%%%%%%%%%%%%%%%%%%%%%%%%%%%%%%
\section{Equivariant quantum K-theory}
\subsection{Preliminaries on equivariant K-theory}\label{sec:preliminaries} In this section we recall some basic facts about the equivariant 
K-theory of a variety with a group action. For an introduction to equivariant K-theory, and more details, see \cite{chriss2009representation}. 

Let $X$ be a smooth projective variety with an 
action of a linear algebraic group $G$. The equivariant 
K-theory ring $\K_G(X)$ is the Grothendieck ring 
generated by symbols $[E]$, where $E \to X$ is a 
$G$-equivariant vector bundle, modulo the relations 
$[E]=[E_1]+[E_2]$ for any short exact sequence 
$0 \to E_1 \to E \to E_2 \to 0$ of equivariant vector bundles. 
The additive ring structure is given by direct sum, and the 
multiplication is given by tensor products of vector bundles. 
Since $X$ is smooth, any $G$-linearized coherent sheaf has 
a finite resolution by (equivariant) vector bundles, and the ring 
$\K_G(X)$ coincides with the Grothendieck group of 
$G$-linearized coherent sheaves on $X$. In particular, any 
$G$-linearized coherent sheaf $\mathcal{F}$ on $X$ 
determines a class $[\mathcal{F}] \in \K_G(X)$. 
An important special case is if $\Omega \subset X$ is a 
$G$-stable subscheme; then its structure sheaf $\cO_\Omega$
determines a class $[\cO_\Omega] \in \K_G(X)$.
The ring $\K_G(X)$ is an algebra over $\K_G(\pt) = \Rep(G)$, the representation ring of 
$G$. If $G=T$ is a complex torus, then this is the Laurent 
polynomial ring 
$\K_T(\pt) = \Z[\ve_1^{\pm 1}, \ldots, \ve_n^{\pm 1}]$ from the previous section.

The (Hirzebruch) $\lambda_y$ class of a vector bundle $E$ is defined by 
\[ \lambda_y(E) := 1 + y [E] + y^2 [\wedge^2 E] + \ldots + y^e  [\wedge^e E] \in \K_G(X)[y] \/. \] 
This class, introduced by Hirzebruch \cite{hirzebruch:topological} in relation to the Grothendieck-Riemann-Roch theorem, may be thought of 
as the K-theoretic analogue of the (cohomological) Chern polynomial 
\[ c_y(E)= 1+ c_1(E) y + \ldots + c_e(E) y^e\] of the bundle $E$. 
The $\lambda_y$ class is multiplicative with respect to short exact sequences, i.e., if 
$0 \to E_1 \to E_2 \to E_3 \to 0$
is such a sequence of vector bundles then 
\[ \lambda_y(E_2) = \lambda_y(E_1) \cdot \lambda_y(E_3) \/; \] 
cf.~\cite{hirzebruch:topological}. 
Since $X$ is proper, the push-forward to a point equals the Euler 
characteristic, or, equivalently, the character of a virtual representation: 
\[\chi(X, \mathcal{F})= \int_X [\mathcal{F}] := \sum_i (-1)^i \ch H^i(X, \mathcal{F}) \/. \] 
(We omit $T$ from the notation.) This defines 
the pairing (with $E,F$ equivariant vector bundles):
\begin{equation}\label{E:pairing} \langle - , - \rangle :\K_T(X) \otimes \K_T(X) \to \K_T(pt); \quad
\langle [E], [F] \rangle := \int_X E \otimes F = \chi(X, E \otimes F) \/.  
\end{equation}

\subsection{Equivariant K-theory of Grassmannians} Fix the Grassmann manifold $\Gr(k,n)$. The Schubert varieties 
$X_\lambda, X^\lambda$ determine classes 
$\cO_\lambda:= [\cO_{X_\lambda}]$ and $\cO^\lambda := [\cO_{X^\lambda}]$
in $\K_{T}(\Gr(k,n))$. The $T$-fixed points $\bfe_\lambda$ give classes which by abuse of notation we still denote by 
$\bfe_\lambda:= [\mathcal{O}_{\bfe_\lambda}]$. The equivariant K-theory 
$\K_{T}(\Gr(k,n))$ is a free module over 
$\K_{\T}(pt)$, with bases given by Schubert classes $\{ \cO_\lambda \}, \{ \cO^\lambda \}$ as $\lambda$ varies in the $k \times (n-k)$ rectangle. 

The (left) action of $\GL_n$ on $Gr(k,n)$ induces an action of the Weyl group $W$ 
\[ \bs{w}: \K_{T}(\Gr(k,n)) \to  \K_{T}(\Gr(k,n)) \/, \quad \kappa  \mapsto \bs{w}.\kappa \/. \]
From definition, $\bs{w_0}.\cO_\lambda = \cO^{\lambda^\vee}$.
This was studied in \cite{knutson:noncomplex}
(in equivariant cohomology), and \cite{harada.landweber.sjamaar:divided,MNS:left}
in K-theory. We refer to this as the {\bf left Weyl group action}. We recall next some basic facts,
 following \cite[\S 5]{MNS:left}. The action $\bs{w}$ is a ring endomorphism of
 $\K_{T}(\Gr(k,n))$ which twists the ground ring $\K_{T}(\pt)$, 
 i.e., for 
$\kappa, \zeta \in  \K_{\T}(\Gr(k,n))$ and $\chi \in \K_{T}(\pt)$
 \[ \bs{w}.(\chi \otimes \kappa \otimes \zeta) = w(\chi) \otimes \bs{w}(\kappa) \otimes \bs{w}(\zeta)\/. \]
The action $\bs{w}$ fixes classes in 
$\K_{\GL_n}(\Gr(k,n))$, i.e.
for $\zeta  \in \K_{\GL_n}(\Gr(k,n))$ and $\kappa  \in \K_{T}(\Gr(k,n))$,
$\bs{w}.(\zeta) = \zeta$. 
For us, the most important examples of classes in $\K_{\GL_n}(\Gr(k,n))$
are the classes of homogeneous vector bundles on $\Gr(k,n)$, 
especially the $\lambda_y$ class of tautological subbundle $\cS$ 
and the (dual of the) quotient bundle $\cQ^\vee$ on $\Gr(k,n)$.

For a fixed point class $\bfe_\lambda \in \K_{T}(\Gr(k,n))$, 
\begin{equation}\label{E:left-fixed} \bs{s}_i.\bfe_\lambda = \bfe_{s_iw_\lambda} \/, \end{equation}
implying also that the pairing from \cref{E:pairing} is invariant under the action of $W$.
In all cases above, by $s_i w_\lambda $ we meant the action of $s_i$ on 
the Young diagram $\lambda$. In practice, this means that if $w_\lambda$ 
is not a Grassmannian permutation then it is replaced by the unique Grassmannian element 
$w_\mu$ which satisfies $s_i \lambda = w_\mu$. 
\begin{footnote}{We warn the reader that the conventions on Schubert classes in this paper are opposite from those in \cite{MNS:left}, and
the operator denoted here by $\bs{\delta}_i$ is denoted by 
$\delta_i^\vee$ in {\em loc.~cit.}}
\end{footnote}

For further use, we record the following expressions for the pairing in the case of the fixed point 
classes in $\K_{T}(\Gr(k,n))$. Recall that $\bfe_\lambda$ is 
the fixed point corresponding to $w_\lambda. \langle f_{1}, \ldots, f_k \rangle$.
Then $\langle \bfe_\lambda, \bfe_\mu \rangle= \mathrm{Eu}_\lambda \delta_{\lambda, \mu}$, where $\Eu_\lambda$ is called the {\bf Euler class at $\bfe_\lambda$}. It may be calculated by
\begin{equation}\label{classEu} \Eu_\lambda = \lambda_{-1}T_{\bfe_\lambda}^*(\Gr(k,n)) = w_\lambda. \prod_{1 \le i \le k < j \le n} (1- \ve_i/\ve_j) \/, \end{equation}
with $T_{\bfe_\emptyset}(\Gr(k,n))$ denoting the tangent space at 
$\bfe_\emptyset$.

Finally, we record some properties we need for the 
intersection pairing.
If $\cO_\lambda, \cO^\mu$ are opposite Schubert classes then 
\begin{equation}\label{E:classical-pairing} \langle \cO_\lambda, \cO^\mu \rangle= \chi(\Gr(k,n),[\cO_{X_\lambda \cap X^\mu}]) = \begin{cases} 1 & \lambda \subset \mu \\ 0 & \textrm{ otherwise} \end{cases}\end{equation}
(This follows because Richardson varieties $X_\lambda \cap X^\mu$ are rational, with rational singularities; \cite{brion:flagv}.) The dual of Schubert classes are the classes of {\bf ideal sheaves} $\mathcal{I}^\mu = [\cO_{X^\mu}(-\partial X^\mu)]$ (see \cite{Buch2002, brion:flagv}): 
\begin{equation}\label{E:ideal-sheaves} \langle \cO_\lambda, \mathcal{I}^\mu \rangle= \delta_{\lambda,\mu} \/.\end{equation}

\subsection{Quantum K-theory}\label{sec:QK}
Let $\K_T(\pt)[\![q]\!]$ be the formal power series 
ring in the variable $q$ with coefficients in $\K_T(\pt)$.
The (small, equivariant) quantum K-ring, defined by Givental and Lee 
\cite{givental:onwdvv, lee:QK},
is a deformation of the usual K-theory ring:
\[ \QK_T(\Gr(k,n)):= \K_{T}(\Gr(k,n)) \otimes_{\K_T(\pt)} \K_T(\pt)[\![q]\!] \/. \]
It is a free $\K_T(\pt)[\![q]\!]$-module with a basis given by 
Schubert classes: 
\[\QK_T(\Gr(k,n))=\bigoplus_{\lambda \subset (n-k)^k} \K_T(\pt)[\![q]\!]\cO_\lambda = \bigoplus_{\lambda \subset (n-k)^k} \K_T(\pt)[\![q]\!]\cO^\lambda \/.\]
Consider now the quantum K-pairing (QK) 
\[ (\!(\cdot, \cdot)\!):\K_T(\Gr(k,n)) \otimes \K_T(\Gr(k,n)) \to 
\K_T(\pt)[\![q]\!] \] 
For any $a, b \in \K_{T}(\Gr(k,n))$,
$(\!(a,b)\!)=\sum_{d \ge 0} \langle a, b \rangle_d q^d$ is the 
generating series of the $2$-point, genus $0$ K-theoretic
Gromov-Witten invariants (KGW) with insertions $a,b$.
The quantum K-multiplication is defined as the unique
product $\star$ satisfying 
\begin{equation}\label{E:qKmetric} (\!(a \star b,c)\!) 
= \sum_{d \ge 0} \langle a,b,c\rangle_d q^d \qquad \in \K_{T}(\pt)[\![q]\!]\/,\end{equation} 
for any $a,b,c \in \K_{T}(\Gr(k,n))$. Here 
$\langle a,b,c\rangle_d$ are the $3$-point, genus $0$ KGW invariants associated to the classes $a,b,c$; 
see \cite{buch.m:qk}. The Weyl group action extends to the equivariant 
quantum K-ring,
leaving the quantum parameter invariant; cf. \cite{MNS:left}. 
Then the QK pairing is $W$-invariant, in the sense 
that for any $w \in W$, and any $a,b \in \QK_{\T}(\Gr(k,n))$,
\[ (\!(\bs{w}.a, \bs{w}.b)\!)= w(\!(a, b)\!) \/.\]
Furthermore, it was proved in \cite{BCLM:euler} that 
the QK pairing is the unique $\K_T(\pt)[\![q]\!]$-bilinear 
map determined by 
\begin{equation}\label{E:QKpairing} (\!(\cO_\lambda, \cO^\mu)\!)= \frac{q^{\mathrm{dist}(\lambda,\mu)}}{1-q}\end{equation}
where $\mathrm{dist}(\lambda,\mu) \in \Z_{\ge 0}$ is 
the minimal degree of a rational curve
joining the opposite Schubert varieties $X_\lambda$ and $X^\mu$.
This degree, which is the same as the smallest degree of $q$ in the 
quantum {\em cohomology} multiplication of the fundamental classes
$[X_\lambda]$ and $[X^\mu]$ has been calculated by 
Postnikov \cite{postnikov:affine} and Fulton-Woodward 
\cite{fulton.woodward:quantum}; see also 
\cite{BCMP:qkfin,buch.m:nbhds} for interpretations in terms
of curve neighborhoods. Explicitly, 
\[\mathrm{dist}(\lambda,\mu)= \min \{d \in \Z_{\ge 0}: \lambda[-d] \subset \mu \}\]
where $\lambda[-d]$ is the partition obtained from $\lambda$ by removing 
its top $d$ rows and leftmost $d$ columns.  
The distance is $0$ if 
$\lambda \subset \mu$, and $\ge 1$ otherwise, showing
that for $q=0$ one recovers the intersection 
pairing from \cref{E:classical-pairing}. 
It follows that the QK pairing is non-degenerate,
hence we may define elements $\mathcal{I}^\lambda_q$ giving a dual basis 
of the Schubert basis:
\begin{equation}\label{E:qideals} (\!( \cO_\lambda, \mathcal{I}^\mu_q)\!) = \delta_{\lambda,\mu} \end{equation}
We refer to these elements as {\bf quantum 
ideal sheaves}. Formulae for $\mathcal{I}^\lambda_q$ were proved in \cite{summers:qideal} (see precursors in \cite{buch.m:qk,BCMP:qkchev}). 
From the definition 
involving KGW invariants it follows that the QK pairing also 
satisfies the Frobenius
property $(\!(a \star b, c)\!) =(\!(a, b \star c)\!)$, and it turns
the QK ring into a Frobenius algebra. 

Recall that any Frobenius algebra $A$ has also naturally the structure of a co-algebra; see e.g. \cite{kock2004frobenius}. Namely, the non-degenerate pairing gives a natural isomorphism with its dual $A\simeq A^*$ and, thus, the multiplication map $m:A\otimes A\to A$ induces a co-product $\Delta:A^*\to A^*\otimes A^*$ which can be pulled back to $A$. In what follows we will make use of both structures, the multiplication and the co-multiplication on $A=QK_T(\Gr(k,n))$. 

Consider the QK structure constants $C_{\lambda,\mu}^{\nu,d} \in \K_T(\pt)$ which can be defined in two dual but equivalent ways:
\[ \cO_\lambda \star \cO_\mu = \sum_{d \ge 0} \sum_{\nu}C_{\lambda,\mu}^{\nu,d} q^d \cO_\nu \qquad\text{and}\qquad
\Delta(\cO_\lambda)= \sum C_{\lambda,\mu}^{\nu,d} \cO_\nu \otimes \mathcal{I}^\mu_q \/. \]
It was proved in \cite{BCMP:qkfin} that the QK multiplication is finite,
i.e., the submodule 
$\K_{T}(\Gr(k,n)) \otimes \K_T[q]$ is stable 
under the QK product $\star$. (Similar statements hold for any flag manifold \cite{ACTI:finiteness,kato:loop}.) More recently, and in the 
non-equivariant case, it was shown in \cite{BCMP:qkpos} 
that $C_{\lambda,\mu}^{\nu,d}$ have predictable signs:
\[ (-1)^{|\lambda|+|\mu|-|\nu|-nd} C_{\lambda,\mu}^{\nu,d} \ge 0 \/.\]
A central problem is to find a manifestly positive combinatorial 
formula for $C_{\lambda,\mu}^{\nu,d}$. 

We will not use this in this paper, but an important property of the 
quantum K ring $\QK_T(\Gr(k,n))$ is that it is equipped with an action of a degeneration of the Hecke algebra, and it is a cyclic Hecke module, generated by the class of the Schubert point $\cO_{(n-k)^k}$; see \cite{MNS:left,gorbounov.korff.mihalcea:YB}.

\subsection{The Yang-Baxter action}
As a particular instance of the gauge/Bethe 
correspondence \cite{nekrasov.shatashvili:supersymmetric} 
(see also 
\cite{faddeev1990lectures,KPSZ:quivers,okounkov:lectures}) 
we proved in \cite{gorbounov.korff.mihalcea:YB} that
the direct sum 
\begin{equation}\label{E:iso}
\bbV_n:=\bigoplus_{0 \le k \le n} \QK_T(\Gr(k,n)) \simeq (\C^2)^{\otimes n} \otimes \K_T(\pt)[\![q]\!] 
\end{equation}
(module isomorphism) admits an action of a Yang-Baxter algebra $\YB$,
generated by the entries of a {\bf monodromy matrix}
\[ \mathbf{T} = \begin{pmatrix} t_{00} & t_{01} \\ t_{10} & t_{11} \end{pmatrix}\]
which satisfies the $\mathbf{R}\mathbf{T}\mathbf{T}=\mathbf{T}\mathbf{T}\mathbf{R}$
relations, with the R-matrix 
\[ \mathbf{R}(z) = \begin{pmatrix} 1 & 0 &0 &0\\
    0& 1-z & z&0\\
    0& 1 & 0&0\\
    0&0&0&1
\end{pmatrix}\]
The module $\bbV_n$ is a highest weight module for $\YB$, and  
the off diagonal matrix entries $t_{ij}$ are raising/lowering operators
\[ t_{10}: \QK_T(\Gr(k,n)) \to \QK_T(\Gr(k+1,n))\/; \quad t_{01}: \QK_T(Gr(k+1,n)) \to \QK_T(Gr(k,n)) \/.\]
Each $\QK_T(\Gr(k,n))$ is a weight space, which is preserved by the diagonal entries:
\[ t_{00},t_{11}: \QK_T(\Gr(k,n)) \to \QK_T(\Gr(k,n)) \/. \]
Explicit formulae for these operators are given in terms of graphical 
calculus of a $5$-vertex integrable lattice model; see \cref{sec:graphical-calculus}. 
The key property for this paper is the identification
\begin{equation}\label{transfer_matrix}
    t_{00}(y) + q t_{11}(y) = \lambda_y(\cQ^\vee) \star
\end{equation}
of the quantum trace of ${\bf T}$ with the QK multiplication operator
by $\lambda_y(\cQ^\vee)$. The quantum trace $t_{00} + q t_{11}$
is also called the (row-to-row) {\bf transfer matrix}.
We will use the transfer matrix to
define and analyze the partition function of the quantum integrable 
system associated to the Yang-Baxter algebra $\YB$.

\subsection{Polynomial representatives}

Schubert classes can be represented as factorial stable Grothendieck polynomials first considered by \cite{mcnamara:factorial} and for which there exists a known determinant formula \cite[Eq. (2.12)]{ikeda2013k}:
\begin{equation}
    G_\lambda(x_1,\ldots,x_k|t)=\frac{\det((x_j|t)^{\lambda_i+k-i}(1+\beta x_j)^{i-1})_{1\le i,j\le k}}{\det(x_j^{k-i})_{1\le i,j\le k}},
\end{equation}
where 
\[
x_j=1+z_j,\quad t_i=1-\ve_i^{-1},\quad(x_j|t)^m=\prod_{i=1}^m(x_j+t_i+\beta x_j t_i)\;.
\]
For most of the paper we will work with the multiplicative variables $z_i,\ve_j$ from \cite[Remark 4.1]{gorbounov.korff.mihalcea:YB} (therefore setting $\beta=-1$) which are more convenient in the geometric setting. In \Cref{sec:cylindric-groth} we will use the other variables $x_i,t_j$ from \cite{gorbounov2017quantum} instead, in order to discuss the definition of Grothendieck polynomials in terms of set-valued tableaux; cf. \cite{Buch2002} and \cite{mcnamara:factorial}.

For a partition $\mu$, set $\ve_\mu=(\ve_{\mu_i+k-i+1})_{1 \le i \le k}$, and consider $\beta=-1$. Then it is well known that the localization coefficient $\langle \cO_\lambda, \bfe_\mu \rangle$ may be calculated by a specialization of the double Grothendieck polynomial (see, e.g., \cite{Oet21}):
\begin{equation}\label{E:locOlambda} 
\langle \cO_\lambda, \bfe_\mu \rangle= G_\lambda(1-\ve_\mu|1-\ve_1^{-1}, \ldots, 1-\ve_n^{-1}) \end{equation}
\section{Graphical calculus and the partition function}\label{sec:graphical-calculus}
In this section we recall the graphical calculus defined by two
of the authors in \cite{gorbounov2017quantum} (see also precursors
in \cite{gorbounov.korff.stroppel,gorbounov.korff:unpublished}), 
and then use it to define the central algebraic 
object in this paper: the partition function associated to 
the transfer operators acting on the quantum K rings of Grassmannians. 
We follow the conventions from \cite{gorbounov.korff.mihalcea:YB}. 

\subsection{The graphical calculus}
We recall next two $5$-vertex models (called the {\em vicious walkers}
and the {\em osculating walkers} in \cite{gorbounov2017quantum})
associated to the Yang-Baxter module $\bbV_n$. We start by 
describing the action of the monodromy matrix operators 
$t_{ij}(z)$. Fix Schubert classes $\cO_\lambda, \cO_\nu$ and let
$J_\lambda=j_1\ldots j_n, J_\nu=j'_1\ldots j'_n$ be the corresponding 
$01$ words. Form the one-row configuration of $n$ vertices 
\begin{center}
\begin{tikzpicture}

\node[anchor=east] at (0.,0.) {$j$};

\node[anchor=south] at (0.25,0.25) {$j_n$};%1st vertex
\node[anchor=north] at (0.25,-0.25) {$j'_n$};

%\node[anchor=south] at (0.75,0.25) {$i_2$};%2nd vertex
%\node[anchor=north] at (0.75,-0.25) {$j_2$};

\node[anchor=south] at (1.75,0.3) {$\cdots$};%4th vertex
\node[anchor=north] at (1.75,-0.3) {$\cdots$};

\node[anchor=south] at (2.75,0.25) {$j_2$};%6th vertex
\node[anchor=north] at (2.75,-0.25) {$j'_2$};

\node[anchor=south] at (3.25,0.25) {$j_1$};%7th vertex
\node[anchor=north] at (3.25,-0.25) {$j'_1$};

\node[anchor=west] at (3.5,0.) {$i$};

\filldraw[black] (0.25, 0) circle (2pt);
\filldraw[black] (0.75, 0) circle (2pt);
\filldraw[black] (1.25, 0) circle (2pt);
\filldraw[black] (1.75, 0) circle (2pt);
\filldraw[black] (2.25, 0) circle (2pt);
\filldraw[black] (2.75, 0) circle (2pt);
\filldraw[black] (3.25, 0) circle (2pt);

\draw[black] (0, 0) -- (3.5, 0);
\draw[black] (0.25, 0.25) -- (0.25, -0.25);
\draw[black] (0.75, 0.25) -- (0.75, -0.25);
\draw[black] (1.25, 0.25) -- (1.25, -0.25);
\draw[black] (1.75, 0.25) -- (1.75, -0.25);
\draw[black] (2.25, 0.25) -- (2.25, -0.25);
\draw[black] (2.75, 0.25) -- (2.75, -0.25);
\draw[black] (3.25, 0.25) -- (3.25, -0.25);
\end{tikzpicture}\;,
\end{center}
For the osculating walkers model, 
each bullet at a single vertex is a placeholder 
for one of the following five possible types of vertices,
\begin{equation}\label{5v}
\text{(I)}\quad\connecting{0}{1}\qquad\text{(II)}\quad
\avoiding{0}{1}\qquad
\text{(III)}\quad\avoiding{1}{0}\qquad
\text{(IV)}\quad\avoiding{0}{0}\qquad
\text{(V)}\quad\avoiding{1}{1}
\;.
\end{equation}
(Note the right to left reading of the $01$ word.)
Each vertex of type (I) in column $i$ (read again right to left) contributes a factor 
$1+z/\ve_{i}$, each vertex of type (II) a 
factor $-z/\ve_{i}$ and (III-V) each contribute 
a factor 1. If there is no possible row configuration consisting 
only of these types of vertices, then the matrix element is zero.
The total weight of this configuration is the product of the weights
of the vertices. Then the coefficient of $\cO_\nu$ in $t_{ij}\cO_\lambda$
is the sum of weights of configurations
\begin{center}
\begin{tikzpicture}
\node[anchor=east] at (0.,0.) {$j$};
%\node[anchor=east] at (-0.5,0.) {$\underset{k=0,1}%{\mathlarger{\sum}}$};

\node[anchor=south] at (0.25,0.25) {$j_n$};%1st vertex
\node[anchor=north] at (0.25,-0.25) {$j'_n$};

%\node[anchor=south] at (0.75,0.25) {$i_2$};%2nd vertex
%\node[anchor=north] at (0.75,-0.25) {$j_2$};

\node[anchor=south] at (1.75,0.3) {$\cdots$};%4th vertex
\node[anchor=north] at (1.75,-0.3) {$\cdots$};

\node[anchor=south] at (3.25,0.25) {$j_2$};%6th vertex
\node[anchor=north] at (3.25,-0.25) {$j'_2$};

\node[anchor=south] at (4.25,0.25) {$j_1$};%7th vertex
\node[anchor=north] at (4.25,-0.25) {$j'_1$};

\node[anchor=west] at (4.5,0.) {$i$};
\node[anchor=east] at (4,0.) {$k$};

\filldraw[black] (0.25, 0) circle (2pt);
\filldraw[black] (0.75, 0) circle (2pt);
\filldraw[black] (2.25, 0) circle (2pt);
\filldraw[black] (2.75, 0) circle (2pt);
\filldraw[black] (3.25, 0) circle (2pt);
%\filldraw[black] (3.75, 0) circle (2pt);
\filldraw[black] (4.25, 0) circle (2pt);

\draw[black] (0, 0) -- (3.5, 0);
\draw[black] (0.25, 0.25) -- (0.25, -0.25);
\draw[black] (0.75, 0.25) -- (0.75, -0.25);
\draw[black] (2.25, 0.25) -- (2.25, -0.25);
\draw[black] (2.75, 0.25) -- (2.75, -0.25);
\draw[black] (3.25, 0.25) -- (3.25, -0.25);
\draw[black] (4, 0) -- (4.5, 0);
\draw[black] (4.25, 0.25) -- (4.25, -0.25);
\end{tikzpicture}
\end{center}
where the sum is over all configurations with the 
specified $01$ words for $\lambda$ and $\nu$; here,
a configuration is a placeholder for its weight, and $k \in \{0,1\}$
stands for the label of any internal edge. In particular,
by \cite[Thm. 1.1]{gorbounov.korff.mihalcea:YB},
\[ (t_{00}(y)+ q t_{11}(y))\cO_\lambda = \lambda_y(\cQ^\vee) \star \cO_\lambda \]
and the left hand side may be calculated using graphical calculus. 
\begin{example}\rm 
Set $k=2$, $n=7$ and $J=1101011$, $J'=1001101$, corresponding to
$\lambda= (3,2)$ and $\nu=(3,1,1)$ respectively. There is a single
configuration calculating the coefficient of $\cO_\nu$ in $t_{10}\cO_\lambda$, given by
%of the weights appearing in the coefficient of $\cO_\nu$ in %$t_{10}\cO_\lambda$ is the (weight of the) configuration
\begin{center}
\begin{tikzpicture}
%%%%%%%LABELLING
\node[anchor=east] at (0.,0.) {$0$};

\node[anchor=south] at (0.25,0.25) {$1$};%1st vertex
\node[anchor=north] at (0.25,-0.25) {$1$};

\node[anchor=south] at (0.75,0.25) {$1$};%2nd vertex
\node[anchor=north] at (0.75,-0.25) {$0$};

\node[anchor=south] at (1.25,0.25) {$0$};%3rd vertex
\node[anchor=north] at (1.25,-0.25) {$1$};

\node[anchor=south] at (1.75,0.25) {$1$};%4th vertex
\node[anchor=north] at (1.75,-0.25) {$1$};

\node[anchor=south] at (2.25,0.25) {$0$};%5th vertex
\node[anchor=north] at (2.25,-0.25) {$0$};

\node[anchor=south] at (2.75,0.25) {$1$};%6th vertex
\node[anchor=north] at (2.75,-0.25) {$0$};

\node[anchor=south] at (3.25,0.25) {$1$};%7th vertex
\node[anchor=north] at (3.25,-0.25) {$1$};

\node[anchor=west] at (3.5,0.) {$1$};
%%%%%%%%%%%%%%%%%%%%%%%%%%%%Drawing
\draw[black] (0, 0) -- (0.2, 0);
\draw[black] (0.25, 0.25) -- (0.25, -0.25);% 1st vertical line
\draw[black] (0.3, 0) .. controls (0.75, 0) .. (0.75, -0.25);

\draw[black] (0.75, 0.25) .. controls (0.75, 0) .. (1, 0);
\draw[black] (1.25, -0.25) .. controls (1.25, 0) .. (1, 0);
\draw[black] (1.25, 0.25) .. controls (1.25, 0) .. (1.7, 0);

\draw[black] (1.75, 0.25) -- (1.75, -0.25);%2nd vertical line
\draw[black] (2.25, -0.25) .. controls (2.25, 0) .. (1.8, 0);
\draw[black] (2.25, 0.25) .. controls (2.25, 0) .. (2.5, 0);
\draw[black] (2.75, -0.25) .. controls (2.75, 0) .. (2.5, 0);
\draw[black] (2.75, 0.25) .. controls (2.75, 0) .. (3, 0);
\draw[black] (3.15, 0) -- (3, 0);

%\draw[black] (3.25, 0.25) -- (3.25, -0.25);%3rd vertical line
\draw[black] (3.25, 0.25) .. controls (3.25, 0) .. (3.5, 0);
\draw[black] (3.25, -0.25) .. controls (3.25, 0) .. (3, 0);
%\draw[black] (3.3, 0) -- (3.5, 0);
\end{tikzpicture}
\end{center}
and this coefficient is 
$
(1+z/\ve_{7})(1+z/\ve_{4})(-z/\ve_{2})(-z/\ve_6)
$
as there is a vertex of type (I) in columns 4 and 7 and a vertex of type (II) in columns 2 and 6 (numbered from right to left). 
%There are many other possible configurations and summing up all their %contributions then gives the full coefficient.
\end{example}

\subsection{The partition function} Fix partitions 
$\lambda, \nu$ in the $k \times (n-k)$ rectangle. The 
{\bf partition function} $Z_\lambda^\nu(z_1, \ldots, z_k|\ve;q)$
is defined by
\begin{equation}\label{E:part-fun} 
\begin{split} Z_\lambda^\nu(z_1, \ldots, z_k|\ve; q)& = \textrm{ coefficient of } \cO_\nu \textrm{ in } \left(\prod_{s=1}^k(t_{00}(z_s)+ q t_{11}(z_s)\right)(\cO_\lambda) \\
& = \textrm{ coefficient of } \cO_\nu \textrm{ in } \left(\lambda_{z_1}(\cQ^\vee) \star \ldots \star \lambda_{z_k}(\cQ^\vee)\right)(\cO_\lambda)
\end{split}
\end{equation}
Since the transfer operators commute (or since the ring $\QK_T(\Gr(k,n))$ is commutative) it follows that 
$Z_\lambda^\nu(z_1, \ldots, z_k|\ve;q)$ is a polynomial
in $\K_T(\pt)[q][z_1, \ldots, z_k]$ which is {\em symmetric}
in the variables $z$.

One may also define the partition function 
using the graphical calculus: 
\begin{equation}\label{E:graph-part-fun}
Z_\lambda^\nu(z_1, \ldots, z_k|\ve;q)=\sum_{\gamma^\lambda_\nu} q^d \op{wt}(\gamma^\lambda_\nu)
\end{equation}
where the sum is over lattice configurations $\gamma_\nu^\lambda$ consisting of $k$ rows and $n$ columns, and such that: 
\begin{itemize} \item the top boundary is given by $\lambda$ and the
bottom boundary is given by $\nu$;
\item the $i$th row from the top corresponds to parameter $z_i$; 
\item the indices on the left and right boundaries coincide;
\item $d$ is the number of $1$'s in the left (or right) boundary.
\end{itemize}
The condition that the left and right boundaries coincide may be interpreted as asking for the configurations on a cylinder with prescribed top and bottom boundaries. The weight $\op{wt}(\gamma^\lambda_\nu)$ of a lattice configuration is defined as the product of all the individual vertex weights occurring in $\gamma^\lambda_\nu$. In particular, lattice configurations containing a vertex not listed in \cref{5v} have weight zero.

%%%%%%%%%%%%%%%%%%%%%%%%%%%%%%%%%%%%%%%%%%%%%%%
\begin{figure}[h!]
\centering
\includegraphics[width=.9\textwidth]{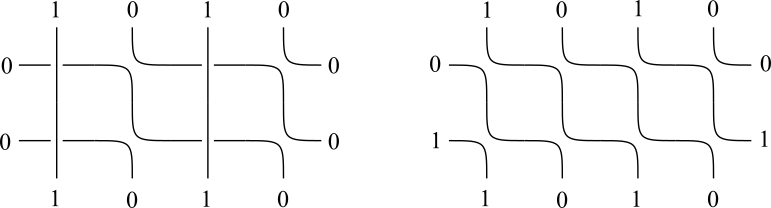} 
\caption{The two possible lattice configurations for $\Gr(2,4)$ and $\lambda=\nu=(1,0)$ with $J_\lambda=0101$.}\label{fig:Gr24example}
\end{figure}

\begin{example}\label{ex:Gr24}\rm
    Let us consider as a simple non-trivial example $\Gr(2,4)$ with $\lambda=\nu=(1,0)$. In this case there are only two possible lattice configurations which are depicted in Figure \ref{fig:Gr24example}. The configuration on the left has weight
    \[
    (1+z_1/\ve_2)(1+z_2/\ve_2)(1+z_1/\ve_4)(1+z_2/\ve_4)\;,
    \] 
    while the one on the right has weight 
    \[
    q(z_1/\ve_2)(z_2/\ve_1)(z_1/\ve_4)(z_2/\ve_3)\;.
    \]
    The power of $q$ occurs since in the last row we have a line carrying a 1-letter crossing the boundary of the cylinder. This corresponds to the action of $t_{11}$ in \cref{E:part-fun}. Thus, the partition function from \cref{E:graph-part-fun} in this simple example reads,
    \begin{equation}\label{Z24}
        Z_\lambda^\nu(z_1, \ldots, z_k|\ve;q)=\frac{qz_1^2z_2^2}{\ve_1\ve_2\ve_3\ve_4}+(1+z_1/\ve_2)(1+z_2/\ve_2)(1+z_1/\ve_4)(1+z_2/\ve_4)\;.
    \end{equation}
\end{example}

\subsection{The dual partition function}
 The level-rank duality is an isomorphism 
$\Theta:\Gr(k,n) \to \Gr(n-k,n)$ depending on 
an ordered (equivariant) basis of $\C^n$. The
induced pull-back 
$\Theta^*:\K_T(\Gr(n-k,n)\to
\K_T(\Gr(k,n))$
satisfies
\[ \Theta^* \lambda_y(\cS) = \lambda_y(\cQ^\vee) \/; \qquad\Theta^*(\ve_i \otimes \cO_\lambda) =
\ve_{n-i+1}^{-1} \otimes \cO_{\lambda^t} \]
where $\lambda^t$ is the transpose of $\lambda$. 
See, e.g., \cite[\S 3.3]{gorbounov.korff.mihalcea:YB}.
In \cite{gorbounov2017quantum} a dual Yang-Baxter action is constructed
 on $\bbV_n$, given by 
a monodromy matrix $\widetilde{\mathbf{T}}=(\widetilde{t}_{ij})_{0 \le i,j \le 1}$ and an R-matrix
$\widetilde{\mathbf{R}}$, and a dual $5$-vertex model,
called the {\em vicious walkers}. By the level-rank duality
the quantum trace corresponds to the QK multiplication by $\lambda_y(\cS)$:
\[ (\widetilde{t}_{00}+q \widetilde{t}_{11})(\cO_\lambda) = \lambda_y(\cS)\star \cO_\lambda \/. \]
The corresponding dual partition function may be defined 
by composition of $n-k$ dual transfer operators:
\begin{equation}\label{E:dual-part-fun} 
\begin{split} \widetilde{Z}_\lambda^\nu(z_1, \ldots, z_{n-k}|\ve; q)& = \textrm{ coefficient of } \cO_\nu \textrm{ in } \left(\prod_{s=1}^{n-k}(\widetilde{t}_{00}(z_s)+ q \widetilde{t}_{11}(z_s)\right)(\cO_\lambda) \\
& = \textrm{ coefficient of } \cO_\nu \textrm{ in } \left(\lambda_{z_1}(\cS) \star \ldots \star \lambda_{z_{n-k}}(\cS)\right)(\cO_\lambda)
\end{split}
\end{equation}
Then $\widetilde{Z}_\lambda^\nu(z_1, \ldots, z_{n-k}|\ve;q) \in \K_T(\pt)[z_1, \ldots, z_{n-k}][q]$ is a symmetric function in the $z$ variables.

There is a graphical calculus representation of the vicious-walker model. By the level-rank duality we will now change the height of the cylinder to $n-k$ instead of $k$. The vertices are
the following:
\begin{equation}
\text{(I')}\quad\avoiding{0}{0}\qquad
\text{(II')}\quad\avoiding{1}{0}\qquad
\text{(III')}\quad\connecting{0}{1}\qquad
\text{(IV)'}\quad\passing{1}{0}\qquad
\text{(V')}\quad\avoiding{0}{1}
\end{equation}
The vertex of type (I') contributes a factor $1+z_r\ve_{c}$ if placed in column $c$ and row $r$ and the vertex of type (II') gives a factor $-z_r\ve_{c}$ while the ones of type (III'-V') give a factor of 1. The rows are read top to bottom, while the columns right to left.
As before, all other types of vertices are assigned the weight zero.
Then one can analogously define the dual partition function
\begin{equation}\label{E:dual-graph-part-fun}
\widetilde{Z}_\lambda^\nu(z_1, \ldots, z_{n-k}|\ve;q)=\sum q^d \op{wt}(\widetilde{\gamma}^\lambda_\nu)
\end{equation}
where the sum is over dual lattice configurations $\widetilde{\gamma}^\lambda_\nu$
satisfying the analogous properties of those in \cref{E:graph-part-fun}. 

\begin{remark}
    Instead of using a model based on lattice paths, the lattice configurations can also be mapped bijectively to dimer configurations on a honeycomb or hexagonal lattice.
\end{remark}

\begin{remark}\rm
While we have stated the second lattice model here for completeness, one can use {\em level-rank duality} \cite[Cor. 4.14]{gorbounov2017quantum}, \cite[Sections 3.3 and 8.3]{gorbounov.korff.mihalcea:YB} to obtain the second lattice model and its partition function from the first. Therefore, we shall henceforth only discuss the lattice model with the vertices from \cref{5v}.
\end{remark}

\section{Products of transfer matrices}
The goal of this section is to study Schubert expansions
of the products of transfer operators which give the partition
functions. The main result is \Cref{prop:expansion-ly}, which sits 
at the foundation of later applications, in 
the calculation of the QK
structure constants, and interpreting the partition function
as an equivariant K-theoretic generalization of Postnikov's toric
Schur polynomials \cite{postnikov:affine}.
We start with the following key lemma. 
\begin{lemma}\label{lemma:noq}
    The coefficients in the expansion of the quantum products
    \[\lambda_{z_1}(\cS) \star \ldots \star \lambda_{z_{n-k}}(\cS)\quad\text{ and }\quad
    \lambda_{z_1}(\cQ^\vee) \star \ldots \star \lambda_{z_{k}}(\cQ^\vee)\]
    into Schubert classes are independent of the quantum parameter $q$.
\end{lemma}

\begin{proof}\label{lemma:expansion} Combinatorially, 
    this is a direct consequence of the graphical calculus. For example, in the case of the multiplications by $\lambda_{z_i}(\cQ^\vee)$, which correspond to the osculating walkers model, each of the lattice paths cannot propagate horizontally by more than one step to the right in a single row. Starting from the $01$-word for the empty partition (i.e., $0^k 1^{n-k}$, read right to left)
    the $1$ strands can only move at most one place to the right, and therefore may reach, but cannot cross, the first column in $k$ steps.
    %all lattice paths start from their $1$ leftmost position (columns %$n$ to $n-k-1$). As there are only $k$ rows the rightmost path %starting in column $n$ can at most reach column 1.
\end{proof}
\begin{remark} One may also prove the previous lemma directly from geometry, by analyzing the maximal powers of $q$ which appear when multiplying Pieri classes; details are left to the reader. See \cite[Prop. 4.1]{gorbounov.korff.mihalcea:YB} and \cite[Thm. 8.3]{BCMP:qkpos}. 
\end{remark}

Recall that a skew partition $\lambda/\mu$ is called a {\bf rook strip} if the corresponding shape has at most one box in each row and column.
\begin{prop}\label{prop:expansion-ly} 
The class $\lambda_{z_1}(\cQ^\vee) \star \ldots \star \lambda_{z_k}(\cQ^\vee)$ has the following expansions.

(a) The expansion of $\lambda_{z_1}(\cQ^\vee) \star \ldots \star \lambda_{z_k}(\cQ^\vee)$ into dual ideal sheaves $\mathcal{I}^\mu$ is given by:
\begin{equation}\label{E:lambda-to-Iupper}
    \lambda_{z_1}(\cQ^\vee) \star \ldots \star \lambda_{z_k}(\cQ^\vee)=
\sum_{\lambda \subset (n-k)^k}G_{\lambda}(1+z_1, \ldots, 1+z_k|1-\ve_1^{-1}, \ldots, 1- \ve_n^{-1})\mathcal{I}^\lambda \/. \end{equation}

(b) The expansion into opposite Schubert classes $\cO^\mu$ is given by:
\begin{equation}\label{E:lambda-to-Oupper}
    \lambda_{z_1}(\cQ^\vee) \star \ldots \star \lambda_{z_k}(\cQ^\vee)=
    \sum_{\mu\subset (n-k)^k} B_{\mu}(z_1,\ldots,z_k|1-\ve^{-1})\;\mc{O}^\mu \quad,
\end{equation}
where $B_\mu(z_1, \ldots, z_k|1-\ve^{-1})$ is the symmetric polynomial 
\[
B_\mu(z_1,\ldots,z_k|1-\ve^{-1}):=\sum_{\lambda}(-1)^{|\lambda/\mu|} G_{\lambda}(1+z_1, \ldots, 1+z_k|1-\ve_1^{-1}, \ldots, 1- \ve_n^{-1})
\]
with the sum running over all partitions $\lambda \subset (n-k)^k$ such that $\lambda/\mu$ is a rook strip.

(c) Define the polynomial 
\begin{equation}\label{Amu}
A_\mu(z|\ve):=\sum_{\lambda}(-1)^{|\lambda/\mu^\vee|}G_{\lambda}(1+z_1, \ldots, 1+z_k|1-\ve_n^{-1},\ldots,1-\ve_1^{-1})
\end{equation}
with the sum running over all partitions $\lambda \subset (n-k)^k$ such that $\lambda/\mu^\vee$ is a rook strip. Then the expansion of $\lambda_{z_1}(\cQ^\vee) \star \ldots \star \lambda_{z_k}(\cQ^\vee)$ into Schubert classes $\cO_\mu$ is given by:
\begin{equation}\label{E:lambda-to-Olower}
    \lambda_{z_1}(\cQ^\vee) \star \ldots \star \lambda_{z_k}(\cQ^\vee)=
    \sum_{\mu\subset (n-k)^k} A_{\mu}(z|\ve_{rev})\;\mc{O}_\mu \quad \/.
\end{equation}
\end{prop}
\begin{proof} By \Cref{lemma:noq} it suffices to prove 
the statement for $q=0$, i.e., by replacing the quantum product 
$\star$ with the multiplication $\cdot$ in the classical ring.
From localization,
\begin{equation}\label{E:locf} 1= \sum_{\mu}\frac{\bfe_{\mu}}{\mathrm{Eu}_{\mu}}\end{equation}
Calculate the multiplication by the fixed point class $\bfe_\mu$: 
\begin{equation}\label{E:classical-e} \lambda_z(\cQ^\vee) \cdot \bfe_\mu = \prod_{j \in [n] \setminus I_\mu} \left(1+ \frac{z}{\varepsilon_j}\right)\bfe_\mu \end{equation}
We now use the following Cauchy-type identity from 
\cite[Prop~2.16 and Cor. ~4.5]{gorbounov2017quantum}:
    \[
    \prod_{i=1}^k\prod_{j\in [n]\backslash I_{\mu}}(1+z_i/\ve_j)=\sum_{\lambda\subset(n-k)^k}G_\lambda(1+z|1-\ve^{-1})(\mc{I}^\lambda,\bfe_\mu)
    \]
with conventions adapted to this paper 
and with shorthand 
$(1+z, 1-\ve^{-1})=(1+z_1, \ldots, 1+z_k|1-\ve_1^{-1}, \ldots, 1-\ve_n^{-1})$.
Using \cref{E:locf} and \cref{E:classical-e}, we calculate:
\begin{eqnarray*}
\lambda_{z_1}(\cQ^\vee)\cdot \ldots \cdot \lambda_{z_k}(\cQ^\vee) & = & 
    \lambda_{z_1}(\cQ^\vee)\cdot \ldots \cdot \lambda_{z_k}(\cQ^\vee)\sum_{\mu}\frac{\bfe_{\mu}}{\mathrm{Eu}_{\mu}} \\
& = & \sum_{\mu}\frac{\bfe_{\mu}}{\mathrm{Eu}_{\mu}} \prod_{i=1}^k\prod_{j\in [n]\backslash I_{\mu}}(1+z_i/\ve_j)\\
&=&\sum_{\mu}\frac{\bfe_{\mu}}{\mathrm{Eu}_{\mu}} \sum_{\lambda}G_{\lambda}(1+z|1-\ve^{-1})(\mathcal{I}^\lambda,\bfe_\mu)\\ 
& = & \sum_{\lambda}G_{\lambda}(1+z|1-\ve^{-1})\mathcal{I}^\lambda   
\end{eqnarray*}    
where the last equality follows again by expanding the (dual) ideal sheaf class 
$\mathcal{I}^\lambda$ into classes of fixed points, using localization. This proves (a). 
We now use the M{\"o}bius inversion to write the ideal sheaf
$\mathcal{I}^\lambda$ into Schubert classes (see, e.g.,~\cite[\S 3.4]{BCMP:qkchev}):
\begin{equation}\label{E:ItoO} \mathcal{I}^\lambda = \sum (-1)^{|\lambda/\mu|} \cO^\mu \end{equation}
where the summation is over $\mu$ such that $\lambda/\mu$ is a rook strip.
Collecting coefficients of $\cO^\mu$ gives part (b) of the proposition.

Finally, to prove part (c), observe that the left Weyl group action by $w_0$ leaves the left hand side invariant (since all bundles are $\GL_n$-homogeneous), while $w_0.\cO^\lambda = \cO_{\lambda^\vee}$ for each $\lambda$. Then the claim follows once we also notice that 
\[ w_0.G_\nu(1+z|1-\varepsilon_1^{-1}, \ldots, 1-\varepsilon_n^{-1})=
G_\nu(1+z|1-\varepsilon_n^{-1}, \ldots, 1-\varepsilon_1^{-1}) \/.\]
\end{proof}
For a partition $\lambda$ in the $k \times (n-k)$ rectangle recall that
$\ve_\lambda=(\ve_{\lambda_i+k-i+1})_{1 \le i \le k}$.
Next we prove the following key result, which will be used to 
calculate quantum K Schubert structure constants. 
\begin{lemma}\label{lemma:loc-ideal} The specialization $z=-\ve_\lambda$ satisfies:
\[(\lambda_{z_1}(\cQ^\vee) \star \ldots \star \lambda_{z_k}(\cQ^\vee))_{z=-\ve_\lambda} = \bfe_\lambda \/.\]
\end{lemma}
\begin{proof} By \Cref{lemma:noq}, it suffices to consider the classical product. 
By localization, 
\[\lambda_{z_1}(\cQ^\vee) \cdot \ldots \cdot \lambda_{z_k}(\cQ^\vee) =
\sum_\nu \frac{\prod_{i=1}^k\prod_{j\in [n]\backslash I_{\nu}}(1+z_i/\ve_j)}{\Eu_\nu}\bfe_\nu \]
If we specialize $(z_1, \ldots, z_k)=-\varepsilon_\lambda$, then
\[ \prod_{i=1}^k\prod_{j\in [n]\backslash I_{\nu}}\left(1-(\varepsilon_\lambda)_i/\ve_j\right) = 0 \]
whenever $\lambda \neq \nu$. Indeed, if 
$I_\lambda \cap ([n] \setminus I_\nu) =\emptyset$ then $I_\lambda \cup ([n] \setminus I_\nu)$ 
must have size $n$, thus it must be $[n]$. 
Then $I_\lambda = I_\nu$, therefore $\lambda=\nu$. Furthermore, using the formula for the Euler class one checks that the coefficient of $\bfe_\lambda$ is equal to $1$, and 
this finishes the proof.\end{proof}
\begin{remark}\label{rmk:skew-groth}\rm
%{\color{blue} Rewrite this after the set valued tableaux section is settled.}
    In the non-equivariant/non-factorial case, i.e.~when $\ve_i=1$, the functions $A_{\mu^\vee}$ in \cref{Amu} are equal to the Lascoux-Sch\"utzenberger $H$-basis in the case of Grassmannians,
    \[
    H_\mu^{(k,n)}(x)
 :=\sum_{\substack{\mu\subset\lambda\subset (n-k)^k\\
                    \lambda/\mu\text{ a rook strip}}}
 (-1)^{|\lambda|-|\mu|}G_\lambda(x),\qquad x_i=1+z_i\;.
\]
Lenart--Robinson--Sottile \cite{lenart2006grothendieck} interpreted these $H$-polynomials as boundary ideal-sheaf classes for Grassmannians. They are also a special case of Yeliussizov's modified skew Grothendieck functions; see \cite{yeliussizov2019symmetric}. The latter are defined for any pair $\beta\subset\alpha$ of partitions as follows: 
\[
 G_{\alpha//\beta}(x)
 :=\sum_{\substack{\rho\subset\beta\\
                   \beta/\rho\text{ a rook strip}}}
 (-1)^{|\beta|-|\rho|}G_{\alpha/\rho}(x).
\]
Here $G_{\alpha/\rho}$ are the stable skew Grothendieck polynomials. Using that $G_{\mu^\vee}=G_{(n-k)^k/\mu}$ one then easily verifies that in the non-equivariant case we have
\begin{equation}
A_\mu(x)=G_{(n-k)^k//\mu}(x)=H^{(k,n)}_{\mu^\vee}(x)\;.
\end{equation}
We will review these definitions in more detail in \Cref{sec:cylindric-groth} when introducing cylindric versions of Grothendieck polynomials, where we also cover the equivariant case.
\end{remark}

\begin{cor}\label{cor:Amu=I} Regard $-z_i$'s as formal Chern roots of $\cS$, i.e., $\prod_{i=1}^k (1-yz_i) = \lambda_y(\cS)$. Then the
following hold:

(a) The polynomial $A_\mu(z|\ve)$
%$A_{\mu}(z_1,\ldots,z_k|1-\ve^{-1}_{rev})$ 
from \cref{Amu} represents the ideal sheaf $\mathcal{I}^\mu$.

(b) The following equality holds: 
%\[ A_\mu(-\varepsilon_\lambda|1-\ve^{-1}_{rev})
\[ A_\mu(-\ve_\lambda|\ve)= \langle\mathcal{I}^\mu, \bfe_\lambda\rangle \]
In particular, $A_\mu(-\ve_\lambda|\ve)=0$ 
whenever $\lambda$ is not included in $\mu$.
\end{cor}
\begin{proof} The hypothesis on the $z_i$'s means that 
 the equivariant localization:
\[ \langle \lambda_y(\cS), \bfe_\lambda\rangle = \left(\prod_{i=1} (1-yz_i)\right)_{-z \mapsto \ve_\lambda} \/.\]
Then $A_\mu(z|\ve)$ 
may also be regarded as the class in K-theory uniquely determined by the localization conditions
\begin{equation}\label{E:Amuclass} \langle A_\mu(z|\ve), \bfe_\lambda \rangle = A_\mu(-\ve_\lambda |\ve) \/\end{equation}
when $\lambda$ varies in $(n-k)^k$. This uses that the localization
map is injective \cite{Nielsen1974}. (Here we abused notation and denoted by the same symbol for the class and the polynomial $A_\mu(z|\ve)$.)
On the other side, from definition, and using duality for ideal sheaves $\langle \cO_\mu, \mathcal{I}^\lambda \rangle = \delta_{\lambda, \mu}$,
\[ A_\mu(z|\ve)= \langle\lambda_{z_1}(\cQ^\vee) \cdot \ldots \cdot \lambda_{z_k}(\cQ^\vee), \mathcal{I}^\mu \rangle
 \/.
\]
By \Cref{lemma:loc-ideal}, $\lambda_{-(\varepsilon_\lambda)_1}(\cQ^\vee) \cdot \ldots \cdot \lambda_{-(\varepsilon_\lambda)_k}(\cQ^\vee) = \bfe_\lambda$, thus
\[ A_\mu(-\varepsilon_\lambda|\ve) = \langle\lambda_{-(\varepsilon_\lambda)_1}(\cQ^\vee) \cdot \ldots \cdot \lambda_{-(\varepsilon_\lambda)_k}(\cQ^\vee), \mathcal{I}^\mu\rangle = \langle\bfe_\lambda,\mathcal{I}^\mu\rangle \/, \]
thus $A_\mu$ and $\mathcal{I}_\mu$ have the same localizations.
Then the two classes must be the same, 
again by injectivity of the localization map, proving (a).
Part (b) follows from the definition of the class $A_\mu$ from 
\cref{E:Amuclass} and the fact that
the ideal sheaf $\mathcal{I}^\mu$ is supported on the 
fixed points $e_\lambda$ such that $\lambda \subset \mu$.
\end{proof}

\begin{remark}\rm One may also analyze a 
quantum version of the specialization from \Cref{lemma:loc-ideal},
when one specializes to roots of the Bethe Ansatz equations; see \cite[Eq. (8.33) ]{gorbounov.korff.mihalcea:YB}.
The Bethe roots 
give the quantum fixed points 
$\bfe_\lambda^q$ of $\QK_T(\Gr(k,n))$, and 
when $q=0$, one recovers the fixed point classes: $(\bfe_\lambda^q)_{q=0}=\bfe_\lambda$; see \cite{gorbounov.korff.mihalcea:YB}.

Consider the simplest case, $\Gr(1,n)=\bP^{n-1}$. The Bethe Ansatz equation for $\bP^{n-1}$ is
\[ \prod_{j=1}^n (1+ x/\ve_j) = q \]
A calculation based on the formulae from \cite{gorbounov.korff.mihalcea:YB} for the on and off-diagonal operators $t_{ii}(z)$ of the monodromy matrix
shows that if $\ve_i^q$ is any solution of the Bethe Ansatz equations, then in $\QK_T(\bP^{n-1})$, 
\[ \lambda_{\ve_i^q}(\cQ^\vee) = \bfe_i^q + q \]
where $\bfe_i^q$ is the quantum fixed point. Therefore the specialization
is not a single quantum fixed-point class.
\end{remark} 

\section{From partition functions to structure constants}
In this section we prove two formulae for the Schubert structure 
constants in the quantum K-ring in terms of the partition function;
cf.~\Cref{thm:alg1} and \Cref{thm:Cramer}. 
In both cases the key is to employ appropriate 
specializations of the variables $z_i$, along with equivariant localization techniques.

\subsection{The geometric definition of the partition function}
We start by rewriting the partition function 
and its dual using the QK pairing. Recall the quantum ideal sheaves defined in \cref{E:qideals}.
\begin{prop} The partition function for the {osculating walkers} on a cylinder with 
$k$ rows and boundary conditions on the top and bottom fixed by 
$\lambda$ and $\nu$, respectively, is the inner product
    \begin{equation}\label{defZ}
    Z^\nu_\lambda(z_1,\ldots,z_{k}|\ve ;q)=\frob{ \cO_\lambda \star \mathcal{I}_q^\nu, \lambda_{z_1}(\cQ^\vee) \star \ldots \star \lambda_{z_{k}}(\cQ^\vee)}
    \end{equation}
The partition function of the {vicious walkers} on a cylinder with $n-k$ rows is:
    \begin{equation}\label{deftildeZ} \widetilde Z^\nu_\lambda(z_1,\ldots,z_{n-k}|\ve;q)=\frob{ \cO_\lambda \star \mathcal{I}_q^\nu, \lambda_{z_1}(\cS) \star \ldots \star \lambda_{z_{n-k}}(\cS)}\end{equation}
Furthermore, if $\lambda=\emptyset$ then both partition functions are independent of $q$.
\end{prop}
\begin{proof} This follows from the definitions in \cref{E:part-fun} and
\cref{E:dual-part-fun}, the Frobenius property of the QK pairing, and that of the quantum ideal sheaves. The independence of $q$ follows 
from \Cref{lemma:noq}, using that $\cO_\emptyset =1$.
\end{proof}
\begin{remark} In what follows we use the class
$\lambda_y(\cQ^\vee)$ and the osculating walkers lattice model, 
but we might just as well use $\lambda_y(\cS)$, and the dual model
of vicious walkers. We leave the details to the reader. 
\end{remark}
The next result relates partition functions to QK structure constants.
\begin{thm}\label{prop:structure-constants} The partition function ${Z}_\lambda^\nu$ satisfies
\[{Z}_\lambda^\nu(z_1, \ldots, z_k|\ve;q) = \sum_{\mu,d} A_{\mu}(z_1, \ldots, z_k|\ve) q^d C_{\lambda,\mu}^{\nu,d}  \]
where $A_{\mu}(z|\ve)$ is defined in 
\cref{Amu}. In particular, 
\[ A_\nu(z|\ve) = Z_\emptyset^\nu(z_1, \ldots, z_k|\ve ;q) \/. \]
\end{thm}
\begin{proof} Combining the definition of $Z_\lambda^\nu$ from
\Cref{defZ}, and \Cref{prop:expansion-ly} we obtain that
\[\begin{split} Z^\nu_\lambda(z_1,\ldots,z_{k}|\ve;q)& =\frob{ \cO_\lambda \star \mathcal{I}_q^\nu, \lambda_{z_1}(\cQ^\vee) \star \ldots \star \lambda_{z_{k}}(\cQ^\vee)} \\ & =
\frob{\cO_\lambda \star \mathcal{I}_q^\nu, \sum_{\mu\subset (n-k)^k} A_{\mu}(z_1,\ldots,z_k|\ve)\;\mc{O}_\mu} \\ & =
\sum_{\mu\subset (n-k)^k} A_{\mu}(z_1,\ldots,z_k|\ve)\;\frob{\cO_\lambda \star \mathcal{I}_q^\nu,\mc{O}_\mu} \\ & =
\sum_{\mu\subset (n-k)^k} A_{\mu}(z_1,\ldots,z_k|\ve)\;\frob{\cO_\lambda \star\cO_\mu, \mathcal{I}_q^\nu}\\
& = \sum_{\mu\subset (n-k)^k} \sum_{d \ge 0} A_{\mu}(z_1,\ldots,z_k|\ve)\;q^d C_{\lambda, \mu}^{\nu,d}\quad \/.
\end{split}
\]
Finally, the formula for $A_\nu$ follows because $C_{\emptyset,\mu}^{\nu,d} = \delta_{d,0}\cdot \delta_{\mu,\nu}$.\end{proof}
We record the following geometric interpretations of the symmetric polynomial giving the partition function.
\begin{cor}\label{cor:part-as-class} The following hold:

(a) Regard the variables $-z_i$ as K-theoretic Chern roots of $\cS$,
i.e., $\prod (1-yz_i) = \lambda_y(\cS)$. Then the partition function $Z_\lambda^\nu(z_1, \ldots, z_k|\ve;q)$ is a polynomial representative for the class
\[ \sum_{\mu,d} C_{\lambda,\mu}^{\nu,d}q^d \mathcal{I}^\mu \/.\]

(b) When specializing $z \mapsto - \ve_\mu$, 
\[ Z_\lambda^\nu(-\ve_\mu|\ve;q)= \frob{\cO_\lambda \star \mathcal{I}^\nu_q, \bfe_\mu} \/. \]
In particular, if one specializes $q=0$,
\[ Z_\lambda^\nu(z_1, \ldots, z_k|\ve;0) = \cO_\lambda \cdot \mathcal{I}^\nu \/. \]

(c) For any partitions $\lambda, \mu$, 
\[ \cO_\lambda \star \bfe_\mu = \sum_\nu Z_\lambda^\nu(-\ve_\mu|\ve;q) \cO_\nu \/.\]
In other words, the partition function records the structure constants in the multiplication
of fixed point classes by Schubert classes. 
\end{cor}
\begin{proof} Part (a) follows from combining \Cref{cor:Amu=I} with 
\Cref{prop:structure-constants}.
The first half of part (b) follows from the specialization in \Cref{lemma:loc-ideal}. The specialization 
at $q=0$ follows because the quantum ideal sheaf $\mathcal{I}_{q=0}^\nu$ 
specializes to the classical ideal sheaf $\mathcal{I}^\nu$, the quantum pairing specializes to the classical
pairing, and then one uses again the injectivity of the localization map in equivariant K-theory.
Finally, (c) follows from the Frobenius
property of the pairing.\end{proof}

\subsection{Two formulae} Next we prove two formulae
for the structure constants $C_{\lambda,\mu}^{\nu,d}$. Both
rely on the specialization $z \mapsto - \ve_\lambda$, 
but the first 
uses the Frobenius structure and localization formula, while 
the second is a determinantal formula
arising from solving a non-homogeneous system of linear 
equations given by an upper triangular matrix. Set
\[ C_{\lambda,\mu}^\nu:= \sum_{d} q^d C_{\lambda, \mu}^{\nu,d} \quad \in \K_T(\pt)[q] \/.\]

\subsubsection{} From the Frobenius property of the QK pairing and 
the definition of the quantum ideal sheaves it follows that 
\[ C_{\lambda,\mu}^\nu= (\!(\cO_\lambda \star \cO_\mu, \mathcal{I}^\nu_q )\!) = (\!(\cO_\lambda \star\mathcal{I}^\nu_q , \cO_\mu )\!) \]
By localization 
\[ \cO_\mu= \sum_\eta \frac{1}{\Eu_\eta} \langle \cO_\mu , \bfe_\eta \rangle \bfe_\eta \]
(with $\Eu_\eta= \lambda_{-1}T^*_{\bfe_\eta}\Gr(k,n)$ denoting the Euler class) thus
\[ C_{\lambda,\mu}^\nu= \sum_\eta \frac{1}{\Eu_\eta} \langle \cO_\mu , \bfe_\eta \rangle
(\!(\cO_\lambda \star\mathcal{I}^\nu_q , \bfe_\eta )\!)
\/. \] 
Combined with \Cref{cor:part-as-class} this implies:
\begin{thm}\label{thm:alg1} The QK structure constants are given by the following equality:
\begin{equation}\label{E:alg1}
 C_{\lambda,\mu}^\nu= \sum_\eta \frac{1}{\Eu_\eta} \langle \cO_\mu , \bfe_\eta \rangle
 Z_\lambda^\nu(-\ve_\eta|\ve;q) \/. \end{equation}
 \end{thm}
Note that the Euler class is combinatorially explicit (cf. \Cref{classEu}) and so are 
the localization coefficients $\langle \cO_\mu, \bfe_\eta\rangle$ (cf. \Cref{E:locOlambda}). Although both 
$\langle \cO_\mu, \bfe_\eta\rangle$ and $Z_\lambda^\nu(-\ve_\eta|\ve;q)$
are Graham positive in the sense of \cite{anderson.griffeth.miller:Kpos}, the Euler class $\Eu_\eta$ twists that positivity by $w_\eta$, thus destroying the positivity of the global formula. (Graham-positivity is still conjectured to hold \cite{buch.m:qk,BCMP:qkpos}.) 

However, in some particular cases, positivity is preserved. For example, 
take $\cO_{(n-k)^k}=\bfe_{(n-k)^k}$ be the class of the Schubert point.
Then
\[ C_{\lambda,(n-k)^k}^\nu = Z_\lambda^\nu(-\ve_{(n-k)^k}|\ve;q) \]
and from the graphical calculus of the partition function one may check that this is appropriately (Graham) positive.

Start from the equations
\[{Z}_\lambda^\nu(z_1, \ldots, z_k|\ve ;q) = \sum_{\mu,d} A_{\mu}(z_1, \ldots, z_k|\ve) q^d C_{\lambda,\mu}^{\nu,d}  \]
Specialize again $z\mapsto-\ve_\eta$, and by \Cref{cor:Amu=I}(b) we obtain
\begin{equation}\label{E:GW-eqs}{Z}_\lambda^\nu(-\ve_\eta|\ve;q) = \sum_{\mu,d} \langle\mathcal{I}^\mu, e_\eta\rangle q^d C_{\lambda, \mu}^{\nu,d} \end{equation}
Since the localization $\langle\mathcal{I}^\mu, e_\eta\rangle=0$ unless 
$\eta \subset \mu$, for each triple $(\lambda,\nu,\eta)$ we 
obtain the system of linear equations:
\begin{equation}\label{E:GW-eqs-v2}{Z}_\lambda^\nu(-\ve_\eta|\ve;q) = \sum_{\eta \subset \mu} \langle\mathcal{I}^\mu, \bfe_\eta\rangle C_{\lambda, \mu}^{\nu} \end{equation}
regarding the QK structure constants as variables 
$C_{\lambda,\mu}^\nu:= \sum_{d} q^d C_{\lambda, \mu}^{\nu,d}$ 
where $\mu \supset \eta$. 
In other words, for each fixed pair $(\lambda, \nu)$ 
we have ${n \choose k}$ equations (accounting for the
choices of $\eta$), depending on ${n \choose k}$ variables 
$C_{\lambda, \mu}^\nu$ (accounting for the choices of $\mu$), and 
with upper triangular (localization) matrix 
\[ L:= \bigl( 
\langle\mathcal{I}^\mu, \bfe_\eta\rangle\bigr)_{\mu,\eta} \/.\]
The diagonal entries are 
$\langle\mathcal{I}^\eta, \bfe_\eta\rangle = \langle\cO^\eta, \bfe_\eta\rangle = 
\lambda_{-1}(N^*_{\bfe_\eta}(X^\eta))$,
where the last term is the Euler class of the conormal bundle of $X^\eta$
at the (smooth) point $\bfe_\eta$. In terms of Grothendieck polynomials  
\begin{equation}\label{E:Groth-loc}\langle\cO^\eta, \bfe_\eta\rangle= 
w_0.\langle\cO_{\eta^\vee}, \bfe_{\eta^\vee}\rangle = 
w_0.G_{\eta^\vee}(1-\ve_{\eta^\vee}|1-\ve^{-1}_1,\ldots,1-\ve^{-1}_n)
\end{equation}
In particular, the determinant of this matrix is the 
product of the Euler classes 
$\lambda_{-1}(N^*_{\bfe_\eta}(X^\eta))$, 
therefore it is non-zero. This system may be solved using Cramer's rule, giving 
the following determinantal formula for the QK structure constants:
\begin{thm}\label{thm:Cramer} The generating function of QK structure constants $C_{\lambda,\mu}^\nu$ is given by the quotient of determinants:
\begin{equation}\label{E:Cramer-form} C_{\lambda,\mu}^\nu = \det L(\mu)/ \prod_\eta \langle\cO^\eta, \bfe_\eta\rangle \end{equation}
where the matrix $L(\mu)$ is obtained from the localization 
matrix $L=\bigl( 
\langle\mathcal{I}^\mu, \bfe_\eta\rangle \bigr)$
by replacing its $\mu$-column by the vector ${Z}_\lambda^\nu(-\ve_\eta|\ve;q)_{\eta \subset (n-k)^k}$. Here the rows and columns are arranged with respect to a fixed total order of partitions, extending the partial order given by inclusion; for example one may use the lexicographic order.
\end{thm}
\begin{remark}\rm 
One may calculate the inverse of the localization matrix
$L$ by  
\[ L^{-1} = \bigl( 
\langle\mathcal{I}^\mu, \bfe_\eta\rangle \bigr)^{-1} = \bigl( 
\frac{1}{\Eu_\eta}\langle \cO_\mu,\bfe_\eta\rangle \bigr) \/.
\]
If we use this and the system of equations from \cref{E:GW-eqs-v2}, one recovers the equalities from \cref{E:alg1}.
\end{remark}

Examples of calculations of structure constants using this result are given next.
\begin{example}\rm We illustrate the formula in \Cref{E:Cramer-form}
for $\Gr(2,4)$ and $\lambda=\nu=(1,0)$. In \Cref{ex:Gr24} we calculated
the partition function $Z_\lambda^\nu$.
The factors 
giving the denominator
in the formula \eqref{E:Cramer-form} are recorded in the following table:
$$
\begin{array}{cc}
\eta & \langle\cO^{\eta}, \bfe_{\eta}\rangle\\
\hline
 (2,2) & 1 \\
 (2,1) & 1-\ve_2/\ve_3 \\
 (2,0) & (1-\ve_1/\ve_2)(1-\ve_1/\ve_3)\\
 (1,1) & (1-\ve_2/\ve_4)(1-\ve_3/\ve_4)\\
 (1,0) & (1-\ve_1/\ve_2)(1-\ve_1/\ve_4)(1-\ve_3/\ve_4)\\
 (0,0) & (1-\ve_1/\ve_3)(1-\ve_2/\ve_3)(1-\ve_1/\ve_4)(1-\ve_2/\ve_4) 
\end{array}
$$
The ideal sheaves $\mathcal{I}^\nu$ are related to the structure sheaves 
$\cO^\nu$ by the M{\"o}bius inversion (\Cref{E:ItoO}). Then one may calculate the Schubert
structure constants $C_{\lambda,\mu}^{\nu,d}$ using \Cref{thm:Cramer}:
\[
\begin{array}{ccc}
\mu & C_{\lambda\mu}^\nu & \lim_{\ve_i\to 1}C_{\lambda\mu}^\nu\\
\hline
 (2,2) & q\;\ve_3 \ve_4/\ve_1\ve_2& q \\
 (2,1) & -q\;\ve_4/\ve_1&-q \\
 (2) & 0 & 0\\
 (1,1) & 0 & 0 \\
 (1,0) & 1-\ve_3/\ve_2&0 \\
 (0,0) & 1&1
\end{array}
\]
Note that in all cases the structure constants are Graham positive,
as conjectured in \cite{BCMP:qkpos}, and proved in {\em loc.~cit.} 
in the non-equivariant case.\end{example}

We end this section by noting that there are also other methods to calculate the structure constants (see, e.g., \cite{buch.m:qk,BCMP:qkchev}) 
but we regard the formulae in this section as a `proof of concept',
using methods which might extend to more general situations.

%%%%%%%%%%%%%%%%%%%%%%%%%%%%%%%%%%%%%%%%%%%%%%%%%%%%%%%%%
\section{Cylindric Grothendieck polynomials}\label{sec:cylindric-groth}

The aim of this section is to state an alternative combinatorial formula for the partition function of the lattice model as a sum of certain generalized Grothendieck polynomials. These are an equivariant and K-theoretic generalization of toric Schur functions considered by Postnikov \cite{postnikov:affine} in the context of (non-equivariant) quantum cohomology of $\Gr(k,n)$. Toric Schur functions are a special subset of more general cylindric Schur functions defined by Gessel and Krattenthaler \cite{gessel1997cylindric} as a weighted sum over semi-standard tableaux on a cylinder. Here we will look instead at set-valued tableaux on a cylinder.

This justifies our terminology of {\em cylindric Grothendieck polynomials}. 
To make the connection between cylindric Grothendieck polynomials and cylindric Schur functions explicit, it will be convenient to change conventions and use throughout this section the multiplicative formal group law and its associated substraction (as in \cite{gorbounov2017quantum}):
\begin{equation}\label{FGL}
x\oplus y=x+y+\beta x y\quad\text{and}\quad x\ominus y=\frac{x-y}{1+\beta y}\;,
\end{equation}
where we treat $\beta$ as a formal variable. In addition, we temporarily adjust the weights of the vertices (I) and (II) in \cref{5v} to
\[
x_i\oplus t_j\quad\text{and}\quad 1+\beta\;x_i\oplus t_j,
\]
respectively. Here $i$ is the lattice row number and $j$ the lattice column number of the vertex. Thus, from now on we write the partition function as a function $Z_\lambda^\nu=Z_\lambda^\nu(x_1,\ldots,x_k|t_1,\ldots,t_n;q)$ of the new variables. Specializing $\beta=0$ we recover the lattice models discussed in \cite{korff2014quantum,gorbounov.korff:unpublished,gorbounov.korff.stroppel} describing (equivariant) quantum cohomology and the partition function then simplifies to cylindric or toric Schur functions.

There is another, purely combinatorial, reason to introduce the variable $\beta$: stable Grothendieck polynomials and their generalizations can be expressed as sums over set-valued tableaux, following Buch \cite{Buch2002}. In order to match lattice configurations with set-valued tableaux, we will expand the weight of a single lattice configuration in powers of $\beta$ and then identify each coefficient in this expansion as weights coming from certain set-valued tableaux. Setting at the end
\begin{equation}\label{x2z}
    x_i=1+z_i,\qquad t_i=1-\ve^{-1}_i,\qquad \beta=-1
\end{equation}
we will recover the conventions used in the other sections of this article and in \cite{gorbounov.korff.mihalcea:YB}. 

\subsection{Partitions and cylindric loops}
We need to recall some notions of cylindric partitions (or loops) and tableaux; see \cite{gessel1997cylindric} although here we shall follow the conventions and notation from \cite{postnikov:affine}.

For a partition $\lambda$ with at most $k$ non-zero parts and satisfying $\lambda_1-\lambda_k\leq n-k$ define its {\bf cylindric loop} $\lambda[d]$ with $d\in\Z$ as the doubly infinite sequence
\begin{equation}\label{cyl_loop}
    \lambda[d]:=(\ldots,\underset{d}{\lambda_k+d+n-k},\underset{d+1}{\lambda_1+d},\ldots,\underset{d+k}{\lambda_k+d},\underset{d+k+1}{\lambda_1+d-n+k},\ldots)\;.
\end{equation}
If $d=0$ this sequence describes a lattice path in $\Z\times\Z$ that gives the outline of the Young diagram of $\lambda$ in the interval $[1,k]$, which is then periodically repeated, $\lambda_{i+k}[0]=\lambda_i[0]-n+k$, i.e. we can interpret $\lambda[0]$ to be a lattice path on the cylinder $\Z^2/\Z(k,-n+k)$. This explains the name `cylindric loop'. If $d\neq 0$ the loop is translated by $d$ steps in the direction of the lattice vector $(1,1)$, that is we set $\lambda_i[d]=\lambda_{i-d}[0]+d$; see Figure \ref{fig:cyl_loop} for an example.

Conversely, if we define a cylindric loop $\rho$ as a weakly decreasing sequence of integers satisfying $\rho_{i+k}=\rho_i-n+k$, then those loops that satisfy  $0\le \rho_k\le\ldots\leq\rho_1\le n-k$ are in bijection with partitions whose Young diagram fits into the rectangle $(n-k)^k$.

\begin{figure}[h!]
\centering
\includegraphics[width=.8\textwidth]{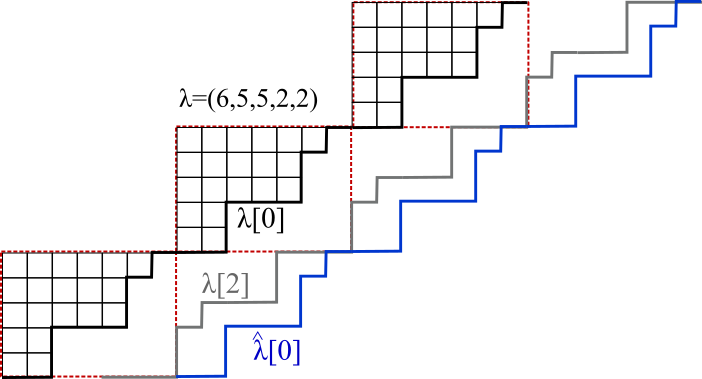} 
\caption{Cylindric loops associated with the partition $\lambda=(6,5,5,2,2)$ for $k=5$ and $n=12$. The blue line indicates the cylindric loop for the partition $\hat\lambda=(n-k+\lambda_1,\ldots,n-k+\lambda_k)$.}\label{fig:cyl_loop}
\end{figure}

Define a {\bf cylindric skew diagram} $\Theta=\theta_1/\theta_2$ as the number of boxes in $\Z\times\Z$ that are enclosed by two cylindric loops, say  $\theta_1$ and $\theta_2$, i.e. the set $\Theta=\{(i,j)\in\Z^2~:~\theta_{2,i}<j\le\theta_{1,i}\}$. 
%see Figure \ref{fig:SVTGr47example} for examples. 
Denote by $|\Theta|=\sum_{i=1}^k(\theta_{1,i}-\theta_{2,i})$ the (finite) number of boxes of the cylindric skew diagram contained in the fundamental strip $[1,k]\times\Z$.
\begin{remark}\rm
    Because of the periodicity condition, we will have in general two choices of representing a cylindric loop or a cylindric skew diagram. Either we restrict to the rows $i=1,\ldots,k$ or to the columns $j=1,\ldots,n-k$. We will usually do the former as we are working in general with $k$ variables.
\end{remark}

A {\bf cylindric set-valued tableau} is a filling of the boxes in rows $1$ to $k$ (or columns $j=1,\ldots,n-k$) with non-empty subsets $S\subset[m]=\{1,\ldots,m\}$, such that the entries weakly increase in rows (left to right) and strictly increase in columns (top to bottom). Moreover, the filling must be compatible with the periodicity condition that the entry in box $(i,j)$ is the same as the entry of the box $(i+rk,j-r(n-k))$ for any $r\in\Z$. Below is the formal definition.
\begin{defn}
    A {\em cylindric set-valued tableau} of shape $\Theta$ is a map $T:\Theta\to 2^{[m]}\backslash\{\emptyset\}$, from $\Theta\subset\Z\times\Z$ to the power set of $[m]$ (excluding the empty set), satisfying
    \begin{equation}\label{cylSVT}
    \max T(i,j)<\min T(i+1,j),\qquad \max T(i,j)\le\min T(i,j+1)
    \end{equation}
    for all adjacent boxes $(i,j),(i+1,j),(i,j+1)\in\Theta$ and
    \begin{equation}\label{cylSVTperiod}
    T(i,j)=T(i+k,j-(n-k))\;.
    \end{equation}
    Given such a tableau $T$ we denote by
    \[
    \#T=\sum_{\substack{(i,j)\in\Theta\\ 1\le i\le k}}\# T(i,j)
    \]
    the total number of entries in the fundamental strip $[1,k]\times\Z\subset\Z\times\Z$.
\end{defn}
If one focuses on the fundamental strip, and drops the condition in \cref{cylSVTperiod}, then one recovers ordinary (non-cylindric) set-valued tableaux. Thus, cylindric set-valued tableaux can be considered as a subset of non-cylindric ones.

We now give an alternative description of (cylindric) set-valued tableaux that will be useful in making the connection with the integrable lattice model. In order to do so, we need to introduce two special families of cylindric skew diagrams first:
\begin{enumerate}
    \item A {\bf cylindric horizontal strip} is defined as a cylindric skew diagram $\Theta$ that has at most one box in each column.

    \item A {\bf cylindric rook strip} is defined as a cylindric skew diagram $\Theta$ that has at most one box in each column {\em and} each row.
\end{enumerate}

Note that the definitions allow for cylindric horizontal and rook strips to be empty. We refer to these as trivial strips.

\begin{lemma}\label{lem:cylSVT}
    Each set-valued cylindric tableau $T$ with entries from $[m]$ and of fixed shape $\Theta=\theta_1/\theta_2$ is in one-to-one correspondence with a sequence $(\tau^{(i)},\rho^{(i)})_{i=0}^m$ of pairs of cylindric loops satisfying
    \begin{itemize}
        \item[(i)] $\rho^{(0)}=\tau^{(0)}=\theta_2$ and $\rho^{(m)}=\tau^{(m)}=\theta_1$;
        \item[(ii)] $\theta_2\subset\rho^{(i)}\subset\tau^{(i)}$ for all $i=1,\ldots,m$;
        \item[(iii)] each $\tau^{(i)}/\rho^{(i)}$ is a cylindric rook strip; 
        \item[(iv)] both, $\tau^{(i)}/\tau^{(i-1)}$ and $\tau^{(i)}/\rho^{(i-1)}$, are cylindric horizontal strips for  $i=1,\ldots,m$.
    \end{itemize}
\end{lemma}
Condition (i), (ii) ensure that none of the cylindric rook strips $\rho^{(i)}$ defines boxes that lie outside the inner boundary loop $\theta_2$. If we would drop these conditions, then the resulting sequences $(\tau^{(i)},\rho^{(i)})_{i=0}^m$ describe cylindric skew shapes $\Theta'$ where the inner boundary can differ by removing a rook strip from $\theta_2$ and we then obtain precisely all cylindric skew shapes entering the definition of the modified cylindric skew Grothendieck polynomial $G_{\tau^{(k)}//\tau^{(0)}}$ defined in \cref{mcsG} below. We will use this fact when computing the partition function of the lattice model in terms of cylindric set-valued tableaux.
\begin{example}\rm
    Consider the set-valued tableau
    \[
    \ytableausetup{notabloids}
\begin{ytableau}
\none & \none & 1 & 1,2 \\
\none & 1 & 2,3 \\
3 & 3
\end{ytableau}
    \]
    which corresponds to the sequence
\begin{multline*}
(\tau^{(0)},\rho^{(0)})=((2,1),(2,1)),\;
(\tau^{(1)},\rho^{(1)})=((4,2,0),(3,2,0)),\;\\
(\tau^{(2)},\rho^{(2)})=((4,3,0),(4,2,0)),\;
(\tau^{(3)},\rho^{(3)})=((4,3,2),(4,3,2))\;.
\end{multline*}
One recovers the tableau from the sequence by putting the letter $i$ in the boxes of the cylindric skew shape $\tau^{(i)}/\rho^{(i-1)}$ for $i=1,2,3$.
\end{example}

\begin{proof}[Proof of \Cref{lem:cylSVT}]
We claim that given a tableau $T$ the corresponding sequence of cylindric loops $(\tau^{(i)},\rho^{(i)})_{i=0}^m$ is obtained as follows: for each $1\le i<m$ the cylindric skew shape $\tau^{(i)}/\theta_2$ are all boxes where the smallest entry is $\le i$ and $\rho^{(i)}/\theta_2$ is then obtained by removing from this set those boxes which also contain $i+1$. Conversely, given a sequence create a tableau $T$ by placing for each $i>0$ the letter $i$ in the boxes of $\tau^{(i)}/\rho^{(i-1)}$.

    Let us start with the simplest case when the entries in the tableau are all singletons, single element sets, in which case $\rho^{(i)}=\tau^{(i)}$ for all $i=0,\ldots,m$. (These are precisely the tableaux which still have nonzero weight when setting $\beta=0$.)  It is a well-known fact that any semi-standard tableau is equivalent to a sequence of horizontal strips, so in the case of $\tau^{(i)}=\rho^{(i)}$  and (non-cylindric) skew shapes there is nothing to prove. 
    
    The extension to cylindric semi-standard tableaux is straightforward, because they can be interpreted as a subset of non-cylindric semi-standard tableaux by restricting to the rows $i=1,\ldots,k$. Because of the periodicity condition in \cref{cylSVTperiod} this determines the tableau and the only difference with the non-cylindric case is that due to a potential overlap of diagrams winding around the cylinder extra constraints are imposed on boxes at the boundary. 
    
    The same argument will apply to set-valued tableaux and, thus, it suffices to look at non-cylindric set-valued tableaux. 
    
    The new aspect we need to prove is that adding additional elements to an already filled box amounts to defining a rook strip. But we claim that this is clear. Namely, for a given tableau $T$ consider the sub-tableau consisting only of those boxes which have entries $\le i$. Because of \cref{cylSVT} the subset of boxes where we can add an additional entry $(i+1)$ are those at the end of a row and at the very bottom of the corresponding column. Hence, boxes containing both, an entry $\le i$ and $(i+1)$, can occur at most once in each row and column. Thus, they form a rook strip $\tau^{(i)}/\rho^{(i)}$. 
    
    The converse is also true: the boxes of any non-empty rook strip must occur at the end of a row and at the bottom of their column, whence the additional entry $(i+1)$ can be added without violating the conditions in \cref{cylSVT}. 

    All entries $(i+1)$ added to new (empty) boxes $\tau^{(i+1)}/\tau^{(i)}$  follow the familiar argument for horizontal strips. Obviously, we cannot add new (empty) boxes in columns where there exist already boxes of the rook strip as this would violate column strictness. Hence, we need to impose that both, $\tau^{(i+1)}/\tau^{(i)}$ and $\tau^{(i+1)}/\rho^{(i)}$ are horizontal strips or, equivalently, that the sets of columns intersected by $\tau^{(i)}/\rho^{(i)}$ (already existing boxes) and $\tau^{(i+1)}/\tau^{(i)}$ (new boxes added) are disjoint. 

    The cylindric case is now obtained by extending from rows $i=1,\ldots,k$ using the periodicity condition in \cref{cylSVTperiod} and ensuring that the conditions in \cref{cylSVT} are still satisfied if there is an overlap. This is precisely the case when $\tau^{(i)}/\tau^{(i-1)}$ is a cylindric horizontal strip and $\tau^{(i)}/\rho^{(i)}$ is a cylindric rook strip.
\end{proof}

\subsection{Cylindric skew Grothendieck polynomials}
We are now ready to define generalisations of known Grothendieck polynomials as sums over cylindric set-valued tableaux.
\begin{defn}
    Define {\em factorial cylindric skew Grothendieck polynomials} by 
\begin{equation}\label{cylG}
    G_{\Theta}(x_1,\ldots,x_m|t)=\sum_{T}\beta^{\#T-|\Theta|}\prod_{\substack{(i,j)\in\Theta\\1\le i\le k}}\prod_{r\in T(i,j)}(x_r\oplus t_{r+j-i}),
\end{equation}
where the sum runs over all set-valued cylindric tableaux $T$ of shape $\Theta=\theta_1/\theta_2$ and we set $t_{i+n}=t_i$ for all $i\in\Z$. 
\end{defn}
N.B. the first product restricts the boxes entering the definition of the polynomial to the fundamental strip $[1,k]\times \Z$ and, thus, there are only a finite number of boxes contributing. 

If $\Theta=\theta_1/\theta_2$ is not a valid cylindric skew shape, i.e. $\theta_2\not\subset\theta_1$ then we set the polynomial equal to zero. Similarly, if a cylindric skew shape $\Theta$ contains columns which have more than $m$ boxes then the condition that the entries in columns must be strictly increasing cannot be satisfied and in this case we also set the corresponding polynomial equal to zero. If $\theta_1=\theta_2$ and the skew shape $\Theta$ is empty, the polynomial is defined to be $1$.

\begin{remark}\rm
    The factorial cylindric skew Grothendieck polynomials are a generalisation of McNamara's factorial (stable) skew Grothendieck polynomials \cite{mcnamara:factorial}: if $\mu\subset\lambda\subset (n-k)^k$ then for the cylindric skew shape $\Theta=\lambda[0]/\mu[0]$ the above definition specialises to the one of McNamara. 
\end{remark}

The second set of polynomials we introduce are generalisations of Yeliussizov's {\em modified skew Grothendieck functions}. Recall from \cite{yeliussizov2019symmetric} that the latter are defined for any pair $\mu\subset\lambda$ of partitions as follows: 
\[
 G_{\lambda//\mu}(x_1,\ldots,x_m)
 =\sum_{\substack{\rho\subset\mu\\
                   \mu/\rho\text{ a rook strip}}}
 \beta^{|\mu/\rho|}G_{\lambda/\rho}(x_1,\ldots,x_m).
\]
Here $G_{\lambda/\rho}$ are {\em stable skew Grothendieck polynomials} defined as sums over (non-cylindric) set-valued skew tableaux.
\begin{defn}
    Let $\theta_2\subset\theta_1$ be two cylindric loops, then we define {\em factorial modified cylindric skew Grothendieck polynomials} as
\begin{equation}\label{mcsG}
    G_{\theta_1//\theta_2}(x_1,\ldots,x_m|t)=\sum_{\substack{\rho\subset\theta_2\\
    \text{cyl. rook strip}}}\beta^{|\theta_2/\rho|}G_{\theta_1/\rho}(x_1,\ldots,x_m|t),
\end{equation}
where the sum ranges over all cylindric loops $\rho$ such that $\theta_2/\rho$ is a {\em cylindric rook strip} as defined previously. For non-valid cylindric skew shapes, $\theta_2\not\subset\theta_1$, we set the polynomial equal to zero. 
\end{defn}
Note that it can happen that while the cylindric skew shape $\theta_1/\theta_2$ has at most $m$ boxes in each column, the shape $\theta_1/\rho$ has more than $m$ boxes in a column. In this case the corresponding cylindric skew Grothendieck polynomial vanishes by definition and, thus, those cylindric rook strips do not contribute to the sum in \cref{mcsG}. 
\begin{remark}\rm
    We recover Yeliussizov's modified Grothendieck polynomials as a special case when the cylindric skew shape $\Theta=\theta_1/\theta_2$ (up to repetition) fits into the bounding box $(n-k)^k$ and setting $t_j=0$ for all $j=1,\ldots,n$. Note, however, that we differ in our conventions from \cite{yeliussizov2019symmetric}: our deformation parameter $\beta$ is related to the one in {\em loc. cit.} by a minus sign. Furthermore, it is not clear at first sight that the polynomials in \cref{cylG} or \cref{mcsG} are symmetric in the $x$-variables, we will show this for certain cases of modified cylindric polynomials using our lattice model below; see \Cref{prop:Z=mcsG}.
\end{remark}

Suppose now that $\lambda,\nu\subset (n-k)^k$. Then we set $\theta_1=\hat\lambda[0]$, the cylindric loop that has entries $\hat\lambda_i[0]=\lambda_i+n-k$ for $i=1,\ldots,k$, and $\theta_2=\nu[d]$ for $0\le d\le k$. Denote by $\tau$ the shift operator for the equivariant parameters, i.e. $\tau(t_i)=t_{i+1}$ and write $G_\lambda(x|\tau^pt)$ etc. for the Grothendieck polynomials with shifted equivariant parameters (this notation is used in the context of factorial Schur polynomials, which we recover when setting $\beta=0$).
\begin{thm}\label{prop:Z=mcsG}
    The partition function $Z^\nu_\lambda$ of our lattice model has the following expansion
    \begin{equation}
        Z^\nu_\lambda(x_1,\ldots,x_k|t;q)=\sum_{d=0}^kq^dG_{\hat\lambda[0]//\nu[d]}(x_1,\ldots,x_k|\tau^kt),
    \end{equation}
    which matches the geometric definition in \cref{defZ} upon setting $x_i=1+z_i$, $t_j=1-\ve^{-1}_j$ and $\beta=-1$.
\end{thm}
We postpone the proof of this theorem to the next subsection as we first need to give an alternative description of the action of the transfer matrix of the lattice model in terms of cylindric loops; see \Cref{lem:transfer_matrix}.

\begin{figure}[h!]
\centering
\includegraphics[width=.9\textwidth]{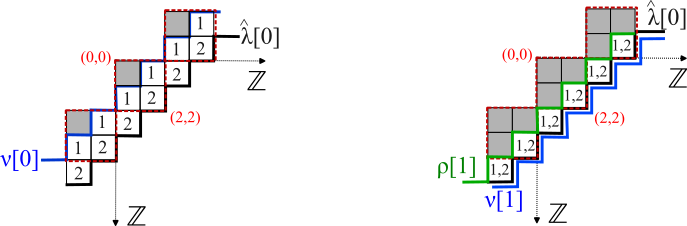} 
\caption{The two possible lattice configurations for $\Gr(2,4)$ and $\lambda=\nu=(1,0)$ from Figure \ref{fig:Gr24example} correspond to the set-valued cylindric skew tableaux shown here.}\label{fig:Gr24exampleSVT}
\end{figure}

\begin{example}\label{ex:Gr24SVT}\rm
    Let us revisit the simple non-trivial example $\Gr(2,4)$ with $\lambda=\nu=(1,0)$ from \Cref{ex:Gr24}. Recall that in this case there are only two possible lattice configurations which are depicted in Figure \ref{fig:Gr24example}. As we will explain below, the left lattice configuration maps onto a unique cylindric skew tableau and in this particular case we have
    \begin{multline*}
         G_{\hat\lambda[0]//\nu[0]}(x_1,x_2|\tau^2t)=G_{(2,2,1)/(1,0,0)}(x_1,x_2|\tau^2t)=(x_1\oplus t_2)(x_2\oplus t_2)(x_1\oplus t_4)(x_2\oplus t_4)\;.
    \end{multline*}
    The second lattice configuration in Figure \ref{fig:Gr24example} has a line carrying a 1-letter crossing the boundary and gives rise to a quantum term. Here we need to compute the modified cylindric skew Grothendieck polynomial $G_{\hat\lambda[0]//\nu[1]}(x_1,x_2|\tau^2t)$. The latter consists of four terms,
    \[
    G_{\hat\lambda[0]//\nu[1]}=1+\beta\, G_{(2,2,1)/(2,2,0)}+\beta\, G_{(2,2,1)/(2,1,1)}+\beta^2 G_{(2,2,1)/(2,1,0)}
    \]
    which arise from the four possible cylindric rook strips $\nu[1]/\rho$. Here the first term corresponds to $\rho=\nu[1]=\hat\lambda[0]$ and the last term to the cylindric rook strip shown on the right in Figure \ref{fig:Gr24exampleSVT}. In this simple example, there is no overlap of the cylindric skew diagrams and, therefore, each single term equals a standard (non-cylindric) skew Grothendieck polynomial, but in general this will not be true.
    
    Due to the occurrence of vertices of type II in the second lattice configuration from Figure \ref{fig:Gr24example} there is not a single cylindric set-valued tableau that corresponds to the configuration, instead one has to expand the weight of the lattice configuration as a polynomial in $\beta$ and match each term to a different tableau. For example, the set-valued tableau shown in Figure \ref{fig:Gr24exampleSVT} on the right belongs to the skew polynomial $G_{(2,2,1)/(2,1,0)}$ and has weight
    \[
    \beta^2 (x_1\oplus t_1)(x_2\oplus t_2)(x_1\oplus t_3)(x_2\oplus t_4)\;.
    \]
    Summing over all possible set-valued skew tableaux originating from all four polynomials and collecting terms, one arrives at
    \[
    G_{\hat\lambda[0]//\nu[1]}=(1+\beta\,x_1\oplus t_1)(1+\beta\,x_2\oplus t_2)(1+\beta\,x_1\oplus t_3)(1+\beta\,x_2\oplus t_4)
    \]
    which matches the weight of the lattice configuration. Thus, we obtain the claimed result for the partition function,
    \[Z_{\lambda}^\nu(x_1,x_2|t;q)=G_{\hat\lambda[0]//\nu[0]}(x_1,x_2|\tau^2t)+q\,G_{\hat\lambda[0]//\nu[1]}(x_1,x_2|\tau^2t)\;.
    \]
    There is no $q^2$ term since the corresponding cylindric skew shape $\hat\lambda[0]/\nu[2]$ is not well-defined and, hence, the corresponding polynomial is by definition equal to zero. Using \cref{x2z} one verifies that the result matches our computation of the partition function in \cref{Z24} from \Cref{ex:Gr24} using the graphical calculus.
\end{example}

As an immediate corollary of \Cref{prop:Z=mcsG} we obtain the following special cases:
\begin{cor}
    (i) For $q=0$ the partition function simplifies to
    \begin{equation}\label{Z=msG}
    Z_{\lambda}^\nu(x_1,\ldots,x_k|t)=G_{\hat\lambda[0]//\nu[0]}(x_1,\ldots,x_k|\tau^kt)
        =G_{((n-k)^k,\lambda)//\nu}(x_1,\ldots,x_k|\tau^kt),
    \end{equation}
    where the factorial modified (non-cylindric) skew Grothendieck polynomials are given by 
    \[
    G_{\lambda//\nu}(x|t)=\sum_{\substack{\rho\subset\nu\\\text{ rook strip}}}\beta^{|\nu|-|\rho|}G_{\lambda/\rho}(x|t),
    \]
    with the sum on the right running over McNamara's \cite{mcnamara:factorial} {\em factorial (stable) skew Grothendieck polynomials}. 
    
    (ii) If we further specialise $\lambda=\emptyset$ then we have the identity
    \begin{equation}\label{Z=H}
    Z^\mu_\emptyset(x_1,\ldots,x_k|t)=G_{(n-k)^k//\mu}(x|\tau^kt)=H^{(k,n)}_{\mu^\vee}(x_1,\ldots,x_k|t_n,\ldots,t_1),
    \end{equation}
    where the $H$-polynomials on the right are defined as
    \begin{equation}\label{Hbasis}
    H^{(k,n)}_\mu(x|t)=\sum_{\substack{\mu\subset\lambda\subset (n-k)^k\\
    \lambda/\mu\text{ rook strip}}}\beta^{|\lambda/\rho|}G_{\lambda}(x|t)\;,
    \end{equation}
    and for $x_i=1+z_i$, $t_j=1-\ve_{j}^{-1}$ and $\beta=-1$ represent {\em boundary ideal sheaf classes}.
\end{cor}
\begin{proof}
(i) We note that in \cref{Z=msG} above we are using the partition $((n-k)^k,\lambda)$ which is outside the bounding box $(n-k)^k$. However, the skew diagrams $((n-k)^k,\lambda)/\rho$, with $\nu/\rho$ a rook strip, can  have at most $k$ boxes in each column. Note that we must have that $\lambda\subset\nu$ otherwise the partition function vanishes. The latter fact can be proved using the graphical calculus of the lattice model and the argument is completely analogous to the one in \cite[Lemmas 3.16 and 3.19]{korff2014quantum} for quantum cohomology albeit with different conventions in {\em loc. cit.}

\cref{Z=msG} then follows from the construction of the cylindric skew diagram $\hat\lambda[0]/\nu[0]$ by noting that the boxes in rows $1,\ldots,k$ have the same entries as the boxes in the columns $1,\ldots,n-k$ and that the content $j-i$ of box $(i,j)$ under a period shift $(i,j)\to(i+k,j-n+k)$ is the same modulo $n$.

To prove (ii) we observe that the Young diagrams of $(n-k)^k/\mu$ and $\mu^\vee$ coincide when rotated by 180 degrees. The rotation leads to a coordinate shift of each box, $(i,j)\to(k+1-i,n-k+1-j)$, and as a result of its content.  Thus, replacing $x_i\to x_{k+1-i}$ and $t_{i+k}\to t_{n+1-i}$ in any tableau of shape $(n-k)^k/\mu$ we obtain the desired identity upon recalling that the modified Grothendieck polynomials are symmetric in the $x$-variables.
\end{proof}
\begin{example}\rm
    Let us consider the example $\Gr(2,5)$ for $q=0$. In the partition function from \cref{Z=H} we choose the boundary conditions on top to be the 01-word of the partition $\lambda=\emptyset$ and on the bottom that of the partition $\mu$. In this case we have that 
\[
Z^\mu_\emptyset(x_1,x_2|t)=H^{(2,5)}_{\mu^\vee}(x_1,x_2|t_5,\ldots,t_1)\;.
\]
 Let us specialise to $\mu=(3,3)$. Then
\[
Z^{(3,3)}_\emptyset(x|t)=H^{(2,5)}_\emptyset(x|t_5,\ldots,t_1)=1+\beta \,G_{1}(x|t_5,\ldots,t_1)=\frac{z_1z_2}{\ve_5\ve_4},
\]
where we have used the variable substitution from \cref{x2z} in the last step. 

On the other hand computing the modified factorial skew Grothendieck polynomial we find
\[
G_{(3,3)//(3,3)}(x|t)=1+\beta\,G_{(3,3)/(3,2)}(x|t)=\frac{z_1z_2}{\ve_2\ve_3}\;,
\]
since the single box of the skew diagram $(3,3)/(3,2)$ occurs in row 2 and column 3 of the Young diagram of $(3,3)$ and, therefore, its content is 1. Shifting the equivariant parameters by $k=2$ we obtain the same result as previously.
\end{example}

    Setting $\beta=0$ we obtain the partition function for the lattice model describing equivariant quantum cohomology \cite{korff2014quantum,gorbounov.korff:unpublished,gorbounov.korff.stroppel} and our formulae simplify considerably. First, note that in the definition of the modified Grothendieck polynomials from \cref{mcsG} only the term with the empty rim hook survives if $\beta=0$, thus, we are led to consider the specialisation of cylindric factorial Grothendieck polynomials at $\beta=0$. The latter become cylindric factorial Schur polynomials as the sum in \cref{cylG} restricts to a sum over cylindric semi-standard tableaux and the multiplicative group law specialises to the additive group law, $x\oplus y|_{\beta=0}=x+y$. 

\begin{cor}\label{cor:toricSchur}
    Setting $\beta=0$ the partition function of the integrable model specialises to
    \[
    Z_{\lambda}^\nu(x_1,\ldots,x_k|t;q)=\sum_{d=0}^kq^d\, s_{\hat\lambda[0]/\nu[d]}(x_1,\ldots,x_k|\tau^kt)
    \]
    where $s_\Theta$ denotes the factorial cylindric or toric Schur function,
    \[
    s_{\Theta}=\sum_{T}\prod_{\substack{(i,j)\in\Theta\\1\le i\le k}}\prod_{r\in T(i,j)}(x_r+ t_{r+j-i})
    \]
    with the sum now running over all cylindric {\em semi-standard} tableaux of shape $\Theta$, i.e. sequences of cylindric loops $(\tau^{(i)})_{i=0}^k$ such that each $\tau^{(i)}/\tau^{(i-1)}$ is a cylindric horizontal strip and $\Theta=\tau^{(k)}/\tau^{(0)}$.
\end{cor}

\begin{remark}\label{rmk:toricSchur}\rm
We note that the definition of the cylindric skew shape $\Theta=\hat\lambda[0]/\nu[d]$ differs from the cylindric skew shape $\nu[d]/\lambda[0]$ considered in the earlier works on toric Schur functions \cite{postnikov:affine,korff2014quantum,gorbounov.korff:unpublished}. However, as already pointed out in the introduction the corresponding toric Schur functions $s_{\hat\lambda[0]/\nu[d]}$ and $s_{\nu[d]/\lambda[0]}$, while of different degree, both allow one to recover the non-equivariant, genus 0, 3-point Gormov-Witten invariants by using the `curious duality' from \cite[Cor.6.9]{postnikov:affine}.
\end{remark}

\subsection{The transfer matrix and cylindric horizontal strips}
In order to prove \Cref{prop:Z=mcsG} we state a combinatorial description of the action of the row-to-row transfer matrix of the lattice model from \cref{transfer_matrix} (albeit now with respect to the Boltzmann weights containing $\beta$; see the comment after \cref{FGL}) in terms of cylindric horizontal strips. Denote by $v_\lambda\in (\C^2)^{\otimes n}$ the basis vector corresponding to the class $\cO_\lambda$ under the isomorphism from \cref{E:iso}.
\begin{lemma}\label{lem:transfer_matrix}
    Let $\lambda\subset(n-k)^k$, then the action of the transfer matrix can be described in terms of cylindric skew diagrams as follows:
    \begin{equation}
        t(x_1)v_\lambda=\sum_{d=0,1}q^d\sum_{\nu}\left(\sum_{\rho\\
        }\beta^{|\nu[d]/\rho|}\prod_{\substack{1\leqslant i\leqslant k\\
        (i,j)\in \lambda^{(1)}[0]/\rho}}x_1\oplus t_{k+1+j-i}\right)v_\nu,
    \end{equation}
    where
    \begin{itemize}
        \item the second sum runs over all $\nu\subset(n-k)^k$ such that $\lambda^{(1)}[0]/\nu[d]$ is a cylindric horizontal strip with $\lambda^{(1)}=(n-k+\lambda_k,\lambda_1,\ldots,\lambda_{k-1})$;
        \item the third sum runs over all cylindric loops $\rho$ such that $\nu[d]/\rho$ is a cylindric rook strip and $\lambda^{(1)}[0]/\rho$ is a cylindric horizontal strip.
    \end{itemize}
\end{lemma}

To illustrate the claim we state a simple example.
\begin{example}\rm
    Consider the Grassmannian $\Gr(2,4)$ and set $\lambda=(2,1)$. Then there are four possible lattice row configurations which are  depicted below. 
\begin{gather}\label{fig:Gr24example2}
\Avoiding{0}{0}{}{0}\hspace{-.55cm}\Connecting{}{1}{}{1}\hspace{-.55cm}
\Avoiding{}{0}{}{0}\hspace{-.55cm}\Connecting{}{1}{0}{1}\qquad
%%%%%%%%%%%%%%%%%%%%%%%%%%
\Avoiding{0}{0}{}{0}\hspace{-.55cm}\Avoiding{}{1}{}{0}\hspace{-.55cm}
\Avoiding{}{0}{}{1}\hspace{-.55cm}\Connecting{}{1}{0}{1}\qquad
%%%%%%%%%%%%%%%%%%%%%%%%%%%%%%%%
\Avoiding{1}{0}{}{1}\hspace{-.55cm}\Connecting{}{1}{}{1}\hspace{-.55cm}
\Avoiding{}{0}{}{0}\hspace{-.55cm}\Avoiding{}{1}{1}{0}\qquad
%%%%%%%%%%%%%%%%%%%%%%%%%%%%%%%%
\Avoiding{1}{0}{}{1}\hspace{-.55cm}\Avoiding{}{1}{}{0}\hspace{-.55cm}
\Avoiding{}{0}{}{1}\hspace{-.55cm}\Avoiding{}{1}{1}{0}
\end{gather}
    The first configuration corresponds to $d=0$, $\nu=(2,1)$ and $\rho=\nu[0]$. The other possible cylindric rook strips $\rho=\mu[0]$ with $\mu=(2,0),(1,1),(1,0)$ violate the condition that $\lambda^{(1)}[0]/\rho$ has to be a cylindric horizontal strip. The associated weight is $(x_1\oplus t_1)(x_1\oplus t_3)$ and matches the weight of the lattice model.

    The second configuration in \cref{fig:Gr24example2} corresponds to $d=0$, $\nu=(2,2)$ and $\rho=\mu[0]$ with $\mu=\nu,(2,1)$. So, in this case there are two cylindric skew diagrams, one has weight $x_1\oplus t_1$ and the other $\beta (x_1\oplus t_1)(x_1\oplus t_3)$. Their sum then matches the weight from the lattice model, $
    (x_1\oplus t_1)(1+\beta\;x_1\oplus t_3)$.

    The last two lattice row configurations in \cref{fig:Gr24example2} both correspond to $d=1$. The third one has $\nu=(0,0)$ and $\rho=\nu[1], (2,1)[0]$. Therefore, we have again two terms which this time combine to the lattice model weight  $
    (x_1\oplus t_3)(1+\beta\;x_1\oplus t_1)$.

    Finally, the last configuration in \cref{fig:Gr24example2}  corresponds to $\nu=(1,0)$ and now there are four possible cylindric rook strips $\rho=\nu[1],(2,2)[0],(0,0)[1],(2,1)[0]$. The weights of these four terms can be combined into the product $(1+\beta\,x_1\oplus t_1)(1+\beta\;x_1\oplus t_3)$ which matches the weight from the lattice model.
\end{example}
\begin{proof}[Proof of \Cref{lem:transfer_matrix}]
    First we note that the transformation $\lambda\to \lambda^{(1)}=(n-k+\lambda_k,\lambda_1,\ldots,\lambda_{k-1})$ corresponds to adding a maximal horizontal strip. Namely, if $\ell(\lambda)<k$ then $\lambda^{(1)}=(n-k,\lambda_1,\ldots,\lambda_{k-1})$ which is the partition with a single box added in each column. If $\ell(\lambda)=k$ then we have in terms of cylindric loops that $\lambda^{(1)}[0]=(\lambda_1-1,\ldots,\lambda_k-1)[1]$, that is, we delete a column of maximal height and then translate the corresponding cylindric loop in the direction of the lattice vector $(1,1)$; compare with \cref{cyl_loop}. Thus, each vertex of type (I) in lattice row $i=1$ and lattice column $j$ amounts to adding a box with content $j-1-k$ in the skew diagram. Since these vertices can occur in adjacent columns this results in a horizontal strip.
    
    For example, if all vertices in \cref{5v} occurring within the lattice row configuration are either of type (I) or (IV), we have that $\lambda^{(1)}[0]/\nu[0]=\lambda^{(1)}[0]/\lambda[0]$ is a maximal horizontal strip and the only possible rook strip is trivial, $\rho=\lambda[0]$.

   Next, we observe that a vertex of type (II) can only occur if in the 01-word $J_\nu$ the respective 0-letter is preceded by a 1-letter, which corresponds to a box at the end of a row and at the bottom of a column in the Young diagram of $\nu$; see \Cref{fig:Young}. In the skew diagram $\lambda^{(1)}[0]/\nu[d]$ these vertices amount to the removal of these boxes and, thus, $\nu[d]/\rho$ must be a rook strip.

   Finally, taking into account that the weights associated with vertices of type (I) and (II) are $x_1\oplus t_j$ and $1+\beta\; x_1\oplus t_j$, respectively, we arrive at the stated formula by interpreting each vertex of type (II) as the `sum of two possible outcomes': either add a box to the rook strip $\nu[d]/\rho$ with weight $\beta\;x_1\oplus t_j$ or do not (with weight 1). Thus, after multiplying out the weights of a single row lattice configuration we obtain the asserted expansion.
\end{proof}

\Cref{lem:transfer_matrix} now allows us to prove \Cref{prop:Z=mcsG} by induction.

\begin{proof}[Proof of \Cref{prop:Z=mcsG}]
    Consider the matrix product $t(x_1)\cdots t(x_{k'})v_\lambda$ with $1\le k'\le k$. We make the slightly more general claim that
    \[
    t(x_1)\cdots t(x_{k'})v_\lambda=\sum_{d=0}^{k'}q^d\sum_{\nu}G_{\lambda^{(k')}[0]//\nu[d]}(x_1,\ldots,x_{k'}|\tau^kt)\,v_\nu\;,
    \]
    where $\lambda^{(k')}=(n-k+\lambda_{k-k'+1},\ldots,n-k+\lambda_k,\lambda_1,\ldots,\lambda_{k-k'})$. Note that $\lambda^{(k')}[0]_i=\lambda[0]_{i-k'}$. Setting $k'=k$ the assertion then follows. The proof proceeds by induction on $k'$ and uses \Cref{lem:transfer_matrix}. Set $k'=1$, then the assertion that
    \[
    t(x_1)v_\lambda=\sum_{d=0}^1 q^d\sum_{\nu}G_{\lambda^{(1)}//\nu[d]}(x_1|\tau^kt)v_\nu
    \]
    is a direct consequence of \Cref{lem:transfer_matrix} and the definition \cref{mcsG}.

    Now, let us assume that for some $1\le k'<k$ the induction hypothesis holds. Relabel the $x$-variables, then
    \[
    t(x_1)\cdots t(x_{k'+1})v_\lambda=\sum_{d=0}^{k'}q^d\sum_{\mu}G_{\lambda^{(k')}[0]//\mu[d]}(x_2,\ldots,x_{k'+1}|\tau^kt)t(x_{1})v_\mu\;.
    \]
    Applying \Cref{lem:transfer_matrix} and \Cref{lem:cylSVT} one deduces that each term resulting from the action of $t(x_{1})$ extends an existing cylindric set-valued tableau $(\tau^{(i)},\rho^{(i)})_{i=1}^{k'+1}$ in the above expansion (the labelling is shifted as we have re-labelled the variables) by an additional pair $(\tau^{(0)},\rho^{(0)})$ which defines a new inner boundary for the cylindric skew diagram. Namely, consider the formula from \Cref{lem:transfer_matrix},
    \[
    t(x_1)v_\mu=\sum_{d'=0,1}q^{d'}\sum_{\nu}\left(\sum_{\rho\\
        }\beta^{|\nu[d']/\rho|}\prod_{\substack{1\leqslant i\leqslant k\\
        (i,j)\in \mu^{(1)}[0]/\rho}}x_1\oplus t_{k+1+j-i}\right)v_\nu,
    \]
    and define $\tau^{(0)}=\nu[d+d']$ and $\rho^{(0)}=\rho[d]$, where $\nu,\rho$ and $d'=0,1$ are the respective quantities appearing in the last equation. The new strip has shape $\mu^{(1)}[d]/\nu[d+d']$. Moreover, because we have added an additional lattice row the cylindric loops in each pair $(\tau^{(i)},\rho^{(i)})$ from the existing set-valued tableau ($i=1,\ldots,k'+1$) need to be shifted, $(\tau^{(i)},\rho^{(i)})\to(\tilde\tau^{(i)},\tilde\rho^{(i)})$, where $\tilde\tau^{(i)}_j=\tau^{(i)}_{j-1}$ and $\tilde\rho^{(i)}_j=\rho^{(i)}_{j-1}$. Since this shifts the content of each box in the resulting skew shape $\lambda^{(k'+1)}/\tau^{(0)}$ by one, it compensates for shifting the variables $x_i\to x_{i+1}$. Thus, the weight of this newly constructed set-valued tableau is the one defined in \cref{cylG} with the equivariant parameters shifted by $k$. Finally, since $\nu[d+d']/\rho[d]|=|\nu[d']/\rho|$ the power  $\beta^{|\nu[d']/\rho|}$ from the added new strip gives the new power in the expansion of $G_{\lambda^{(k'+1)}[0]//\nu[d+d']}$ into cylindric skew Grothendieck polynomials from \cref{mcsG}, while the power $\beta^{|\mu[d]/\rho^{(1)}|}$ matches the number of boxes which contain both, a 1 and a 2-letter in the new set-valued tableau.
\end{proof}

%%%%%%%%%%%%%%%%%%%%%%%%%%%%%%%%%%%%%%%%%%%%%%%%%%%%%%%%%

\section{Other variations of partition functions} \Cref{prop:structure-constants} and \Cref{thm:alg1} give a precise meaning to the fact 
that the partition function 
$Z_\lambda^\nu(z_1, \ldots, z_k|\ve;q)$ is adapted to 
calculations of Schubert 
structure constants of the quantum K-ring $\QK_T(\Gr(k,n))$. 
There are other versions,
adapted to calculate KGW invariants, or to 
calculate generating functions of classes of 
two-point curve neighborhoods (see below). In this section, 
we briefly discuss these variants, and relate them to the
partition functions $Z_\lambda^\nu$.
An interesting question we do not address 
is to find the graphical calculus that 
gives these variants directly. However, we note that 
these variants are no longer polynomials, but power series
in the quantum variable $q$.

\subsection{A partition function adapted to two-point curve neighborhoods}
We start by recalling some relevant facts 
about two-point curve neighborhoods, following 
\cite{bcmp:projgw}, see also \cite{BCMP:qkfin,buch.m:nbhds}.
Fix partitions $\lambda, \nu$ and opposite Schubert varieties 
$X_\lambda, X^\nu$. 
For a fixed degree $d \ge 0$, the corresponding {\bf Gromov-Witten variety}
$\mathrm{GW}_d(X_\lambda, X^\nu)$ is the subvariety
of the moduli space of stable maps $\Mb_{0,3}(\Gr(k,n),d)$ defined by 
\[  \mathrm{GW}_d(X_\lambda, X^\nu) = \ev_1^{-1}(X_\lambda) \cap  \ev_2^{-1}(X^\nu) \/. \]
Here $\ev_i: \Mb_{0,3}(\Gr(k,n),d) \to \Gr(k,n)$ are the evaluation maps.
The $2$-point curve neighborhood of $X_\lambda, X^\nu$ is 
the image of the GW variety
inside the Grassmannian:
\[ \Gamma_d(X_\lambda, X^\nu):= \ev_3(\mathrm{GW}_d(X_\lambda, X^\nu)) \subset \Gr(k,n) \/. \]
It was proved in \cite{BCMP:qkfin} that the GW variety is either empty, or irreducible, rationally connected, with rational singularities. Moreover, the evaluation map $\ev_3: \mathrm{GW}_d(X_\lambda, X^\nu) \to  \Gamma_d(X_\lambda, X^\nu)$ is cohomologically trivial \cite[Thm.~4.1]{bcmp:projgw}. (See also \cite{BCMP:qkpos} for a unified account.) Furthermore, the two-point curve neighborhood is a projected Richardson variety \cite[\S 5]{bcmp:projgw}. The projected Richardson varieties, named {\bf positroid varieties} in \cite{knutson.lam.speyer:juggling}, are irreducible, rational, 
and with rational singularities (cf.~\cite{knutson.lam.speyer:projrich}). 
Combining all this with the positivity proved in \cite{anderson.griffeth.miller:Kpos} yields the following:

\begin{prop}\cite[Rmk.~4.4]{bcmp:projgw} The following identity holds: 
\begin{equation}\label{E:Gammad-exp} \cO_{\Gamma_d(X_\lambda, X^\nu)} = \sum_\mu \langle \cO_\lambda, \cO^\nu, \mathcal{I}^\mu \rangle_d \cO_\mu \end{equation}	
and the Gromov-Witten invariants are Graham positive, in the sense of \cite{anderson.griffeth.miller:Kpos}. 
\end{prop}
We ask the reader to observe the similarity of this class to the expression from \Cref{cor:part-as-class}, which gives the class represented by the partition function $Z_\lambda^\nu$. Motivated by this, 
we now consider a variant of the partition function which recovers 
the expansion from \Cref{E:Gammad-exp}. Set 
\begin{equation}\label{def:ZGamma} 
Z_{\Gamma,\lambda}^\nu(z_1, \ldots, z_k|\ve;q):= 
\frob{\cO_\lambda \star \cO^\nu, \lambda_{z_1}(\cQ^\vee) \star \cdots\star  \lambda_{z_k}(\cQ^\vee)} \/.
\end{equation}
\begin{prop}\label{prop:part-function-Gamma} 
The following hold:

(a) The partition function $Z_{\Gamma,\lambda}^\nu(z_1, \ldots, z_k|\ve;q)$ 
may be expressed in terms of KGW invariants:

\begin{equation} \begin{split} Z_{\Gamma,\lambda}^\nu(z_1, \ldots, z_k|\ve;q)& = \sum_{\mu}G_{\mu}(1+z|1-\ve^{-1})\frob{\cO_\lambda \star \cO^\nu, \mathcal{I}^\mu} \\
& = \sum_{\mu,d} G_{\mu}(1+z|1-\ve^{-1}) \langle \cO_\lambda, \cO^\nu, \mathcal{I}^\mu \rangle_d q^d \end{split} \end{equation}

(b) The partition functions $Z$ and $Z_\Gamma$ are related by: 
\[ \begin{split} Z_{\Gamma,\lambda}^\nu(z_1, \ldots, z_k|\ve;q)& = \sum_\eta \frob{\cO^ \nu, \cO_\eta} Z_\lambda^\eta (z_1, \ldots, z_k|\ve;q) \\
&= \sum_\eta \frac{q^{\mathrm{dist}(\eta,\nu)}}{1-q}Z_\lambda^\eta(z_1, \ldots, z_k|\ve;q) \end{split} \]
where $\mathrm{dist}(\eta,\nu)$ is the distance between $X^\nu$ and $X_\eta$ from \cref{E:QKpairing}.
\end{prop}
\begin{proof} Part (a) follows the same outline as in the proof of \Cref{prop:structure-constants},
using the expansion of $\lambda_{z_1}(\cQ^\vee) \star \ldots  \lambda_{z_k}(\cQ^\vee)$ into opposite ideal sheaves from \Cref{prop:expansion-ly}. Part (b) follows because
\[ \cO^\nu = \sum_\eta \frob{\cO^\nu, \cO_\eta} \mathcal{I}^\eta_q = \sum_\eta \frac{q^{\mathrm{dist}(\eta,\nu)}}{1-q}\mathcal{I}^\eta_q \/,\]
where the last equality is proved in the main theorem from \cite{BCLM:euler}. 
\end{proof}

\begin{cor} Regard the variables $-z_i$ as K-theoretic Chern roots of $\cS$,
i.e., $\prod (1-yz_i) = \lambda_y(\cS)$. Then 
\[ Z_{\Gamma, \lambda}^\nu(z|\ve;q) = \sum_{d \ge 0} q^d [\cO_{\Gamma_d(X_\lambda, X^\nu)}] \]
as elements in $\K_T(\Gr(k,n))[\![q]\!]$.
\end{cor}
\begin{proof} Under the given identification of variables
$z_i$, \Cref{prop:part-function-Gamma}(a) states that
\[Z_{\Gamma, \lambda}^\nu(z|\ve;q) =\sum_{\mu, d} \langle \cO_\lambda, \cO^\nu, \mathcal{I}^\mu \rangle_d q^d \cO_\mu \]
Then the claim follows from \Cref{E:Gammad-exp}.  
\end{proof}

\subsection{A partition function adapted to K-theoretic GW invariants}
We record below another variant which is adapted to 
calculations of KGW invariants 
$\langle \cO_\lambda, \cO_\mu,\cO_\nu \rangle_d$. 
The proofs are only outlined, with 
details left to the reader. 
\begin{prop} Consider the variant of the partition function:
\[ Z_{\mathrm{GW},\lambda}^\nu(z_1, \ldots, z_k|\ve;q):= \frob{\cO_\lambda \star \cO_\nu, \lambda_{z_1}(\cQ^\vee) \star \cdots\star  \lambda_{z_k}(\cQ^\vee)} \]
Then the following hold:

(a) The partition function may be expressed in terms of KGW invariants:
\begin{equation} \begin{split} Z_{\mathrm{GW},\lambda}^\nu(z_1, \ldots, z_k|\ve;q)& =\sum_{\mu} A_{\mu}(z|\ve)\;\frob{\cO_\lambda \star\cO_\nu, \cO_\mu} \\
& = \sum_{\mu,d} A_{\mu}(z|\ve) \langle \cO_\lambda, \cO_\nu, \cO_\mu \rangle_d q^d \end{split} \end{equation}
where the coefficients $A_{\mu}(z|\ve)$ are those from \Cref{prop:expansion-ly}.

(b) The partition function $Z_{\mathrm{GW},\lambda}^\nu$ may be expressed in terms of 
the partition function $Z$ by: 
\begin{equation}\begin{split} Z_{\mathrm{GW},\lambda}^\nu(z_1, \ldots, z_k|\ve;q)& = 
\sum_\eta \frob{\cO_\nu, \cO_\eta} Z_\lambda^\eta (z_1, \ldots, z_k|\ve;q) \\
%& = \sum_{\eta,d} \langle \cO_\nu, \cO_\eta}_d q^d Z_\lambda^\eta (z_1, \ldots, z_k;q)\\
& = \sum_{\eta} \sum_\alpha c_{\alpha,\eta} \frac{q^{\mathrm{dist}(\nu,\alpha)}}{1-q} Z_\lambda^\eta (z_1, \ldots, z_k|\ve;q) \end{split}
\end{equation}
where $c_{\alpha,\eta}$ are the coefficients from the expansion $\cO_\eta = \sum c_{\alpha,\eta} \cO^\alpha$.\end{prop}
\begin{proof} Part (a) follows from expanding 
$\lambda_{z_1}(\cQ^\vee) \star \ldots \star \lambda_{z_k}(\cQ^\vee)$ into 
Schubert classes $\cO_\mu$, then using \Cref{prop:expansion-ly}. 
Part (b) relies on a similar argument to that in \Cref{prop:part-function-Gamma}, 
using that $\cO_\nu = \sum_\eta \frob{\cO_\nu, \cO_\eta} \mathcal{I}^{\eta}_q$, expanding 
$\cO_\eta$ into opposite Schubert classes, and finally using the formula for the 2-point invariants
$\frob{\cO_\nu, \cO^\alpha}$.  
\end{proof}

%%%%%%%%%%%%%%%%%%%%%%%%%%%%%%%%%
%%%%%%%%%%%%%%%%%%%%%%%%%%%
\bibliographystyle{halpha}
\bibliography{BIB.bib}

@book {Baxter:book,
    AUTHOR = {Baxter, Rodney J.},
     TITLE = {Exactly solved models in statistical mechanics},
 PUBLISHER = {Academic Press, Inc. [Harcourt Brace Jovanovich, Publishers],
              London},
      YEAR = {1982},
     PAGES = {xii+486},
      ISBN = {0-12-083180-5},
   MRCLASS = {82-02 (82A05 82A68)},
  MRNUMBER = {690578},
MRREVIEWER = {J. Groeneveld},
}

@article{faddeev1990lectures,
  title={Lectures on quantum inverse scattering method},
  author={Faddeev, LD},
  journal={New Problems, Methods and Techniques in Quantum Field Theory and Statistical Mechanics. Series on Advances in Statistical Mechanics},
  volume={6},
  pages={7--54},
  year={1990}
}

@article{gorbounov2017quantum,
	author = {Gorbounov, Vassily and Korff, Christian},
	journal = {Advances in Mathematics},
	pages = {282--356},
	publisher = {Elsevier},
	title = {Quantum integrability and generalised quantum {S}chubert calculus},
	volume = {313},
	year = {2017}}

@book{chriss2009representation,
  title={Representation theory and complex geometry},
  author={Chriss, Neil and Ginzburg, Victor},
  year={2009},
  publisher={Springer Science \& Business Media}}

@book {hirzebruch:topological,
    AUTHOR = {Hirzebruch, Friedrich},
     TITLE = {Topological methods in algebraic geometry},
    SERIES = {Classics in Mathematics},
 PUBLISHER = {Springer-Verlag, Berlin},
      YEAR = {1995},
     PAGES = {xii+234},
      ISBN = {3-540-58663-6},
   MRCLASS = {57-02 (01A75 14-02)},
  MRNUMBER = {1335917},
}

@incollection{brion:flagv,
	Author = {Brion, Michel},
	Booktitle = {Topics in cohomological studies of algebraic varieties},
	Pages = {33--85},
	Publisher = {Birkh\"auser, Basel},
	Series = {Trends Math.},
	Title = {Lectures on the geometry of flag varieties},
	Year = {2005}}

@article {givental:onwdvv,
    AUTHOR = {Givental, Alexander},
     TITLE = {On the {WDVV} equation in quantum {$K$}-theory},
      NOTE = {Dedicated to William Fulton on the occasion of his 60th
              birthday},
   JOURNAL = {Michigan Math. J.},
  FJOURNAL = {The Michigan Mathematical Journal},
    VOLUME = {48},
      YEAR = {2000},
     PAGES = {295--304},
      ISSN = {0026-2285},
   MRCLASS = {14N35 (19E08 53D45)},
  MRNUMBER = {MR1786492 (2001m:14078)},
MRREVIEWER = {Gary P. Kennedy},
}

@article {lee:QK,
    AUTHOR = {Lee, Y.-P.},
     TITLE = {Quantum {$K$}-theory. {I}. {F}oundations},
   JOURNAL = {Duke Math. J.},
  FJOURNAL = {Duke Mathematical Journal},
    VOLUME = {121},
      YEAR = {2004},
    NUMBER = {3},
     PAGES = {389--424},
      ISSN = {0012-7094},
     CODEN = {DUMJAO},
   MRCLASS = {14N35 (19E08 53D45 55N15)},
  MRNUMBER = {MR2040281 (2005f:14107)},
MRREVIEWER = {Andrew Kresch},
}

@article{buch.kresch.ea:gromov-witten,
	author= {A.~S. Buch and A.~Kresch and H.~Tamvakis}, 
	Title = {Gromov-{W}itten invariants on {G}rassmannians}, 
  	journal={J. Amer. Math. Soc.},
	volume={16}, 
	year={2003}, 
	number={4}, 
	pages={901--915(electronic)}
}

@article{BCMP:qkpos,
	author= {A.~S. Buch and Chaput, Pierre-Emmanuel and Mihalcea,
              Leonardo C. and Perrin, Nicolas}, 
	Title = {Positivity of minuscule quantum {K} theory}, 
  	journal={arxiv:2205.08630}, 
	year={2022},
}

@article {buch.m:nbhds,
 AUTHOR = {Buch, Anders S. and Mihalcea, Leonardo C.},
     TITLE = {Curve neighborhoods of {S}chubert varieties},
   JOURNAL = {J. Differential Geom.},
  FJOURNAL = {Journal of Differential Geometry},
    VOLUME = {99},
      YEAR = {2015},
    NUMBER = {2},
     PAGES = {255--283},
      ISSN = {0022-040X},
   MRCLASS = {14L30 (14M17 14N35)},
  MRNUMBER = {3302040},
MRREVIEWER = {Pierre-Emmanuel Chaput},
       URL = {http://projecteuclid.org/euclid.jdg/1421415563},
}

@article {BCMP:qkfin,
    AUTHOR = {Buch, Anders S. and Chaput, Pierre-Emmanuel and Mihalcea,
              Leonardo C. and Perrin, Nicolas},
     TITLE = {Finiteness of cominuscule quantum {$K$}-theory},
   JOURNAL = {Ann. Sci. \'{E}c. Norm. Sup\'{e}r. (4)},
  FJOURNAL = {Annales Scientifiques de l'\'{E}cole Normale Sup\'{e}rieure. Quatri\`eme
              S\'{e}rie},
    VOLUME = {46},
      YEAR = {2013},
    NUMBER = {3},
     PAGES = {477--494 (2013)},
      ISSN = {0012-9593},
   MRCLASS = {19Exx (14M15 14N35 81R99)},
  MRNUMBER = {3099983},
MRREVIEWER = {Dragos Nicolae Oprea},
       DOI = {10.24033/asens.2194},
       URL = {https://doi.org/10.24033/asens.2194},
}

@article {BCLM:euler,
    AUTHOR = {Buch, Anders S. and Chung, Sjuvon and Li, Changzheng and
              Mihalcea, Leonardo C.},
     TITLE = {Euler characteristics in the quantum {$K$}-theory of flag
              varieties},
   JOURNAL = {Selecta Math. (N.S.)},
  FJOURNAL = {Selecta Mathematica. New Series},
    VOLUME = {26},
      YEAR = {2020},
    NUMBER = {2},
     PAGES = {Paper No. 29, 11},
      ISSN = {1022-1824},
   MRCLASS = {14N35 (14M15 14N15 19E08)},
  MRNUMBER = {4083680},
MRREVIEWER = {Guillermo Corti\~{n}as},
       DOI = {10.1007/s00029-020-00557-7},
       URL = {https://doi.org/10.1007/s00029-020-00557-7}
}

@article{ACTI:finiteness,
  title={On the finiteness of quantum {K}-theory of a homogeneous space},
  author={Anderson, David and Chen, Linda and Tseng, Hsian-Hua and Iritani, Hiroshi},
  journal={arXiv preprint arXiv:1804.04579},
  year={2018}
}

@article{kato:loop,
  title={Loop structure on equivariant {K}-theory of semi-infinite flag manifolds},
  author={Kato, Syu},
  journal={arXiv preprint arXiv:1805.01718},
  year={2018}
}

@article {buch.m:qk,
    AUTHOR = {Buch, Anders S. and Mihalcea, Leonardo C.},
     TITLE = {Quantum {$K$}-theory of {G}rassmannians},
   JOURNAL = {Duke Math. J.},
  FJOURNAL = {Duke Mathematical Journal},
    VOLUME = {156},
      YEAR = {2011},
    NUMBER = {3},
     PAGES = {501--538},
      ISSN = {0012-7094},
   MRCLASS = {14N35 (14C35 14M15 19E08)},
  MRNUMBER = {2772069},
MRREVIEWER = {Hsian-Hua Tseng},
       DOI = {10.1215/00127094-2010-218},
       URL = {https://doi.org/10.1215/00127094-2010-218},
}

@article {bcmp:projgw,
    AUTHOR = {Buch, Anders S. and Chaput, Pierre-Emmanuel and Mihalcea,
              Leonardo C. and Perrin, Nicolas},
     TITLE = {Projected {G}romov-{W}itten varieties in cominuscule spaces},
   JOURNAL = {Proc. Amer. Math. Soc.},
  FJOURNAL = {Proceedings of the American Mathematical Society},
    VOLUME = {146},
      YEAR = {2018},
    NUMBER = {9},
     PAGES = {3647--3660},
      ISSN = {0002-9939},
   MRCLASS = {14N35 (14M15 14M20 14M22 14N15 19E08)},
  MRNUMBER = {3825822},
       DOI = {10.1090/proc/13839},
       URL = {https://doi.org/10.1090/proc/13839},
}

@article {buch:groth-classes,
    AUTHOR = {Buch, Anders Skovsted},
     TITLE = {Grothendieck classes of quiver varieties},
   JOURNAL = {Duke Math. J.},
  FJOURNAL = {Duke Mathematical Journal},
    VOLUME = {115},
      YEAR = {2002},
    NUMBER = {1},
     PAGES = {75--103},
      ISSN = {0012-7094,1547-7398},
   MRCLASS = {14C35 (05E10 19E08)},
  MRNUMBER = {1932326},
MRREVIEWER = {Gary\ P.\ Kennedy},
       DOI = {10.1215/S0012-7094-02-11513-0},
       URL = {https://doi.org/10.1215/S0012-7094-02-11513-0},
}

@article {MNS:left,
    AUTHOR = {Mihalcea, Leonardo C. and Naruse, Hiroshi and Su, Changjian},
     TITLE = {Left {D}emazure-{L}usztig operators on equivariant (quantum)
              cohomology and {K}-theory},
   JOURNAL = {Int. Math. Res. Not. IMRN},
  FJOURNAL = {International Mathematics Research Notices. IMRN},
      YEAR = {2022},
    NUMBER = {16},
     PAGES = {12096--12147},
      ISSN = {1073-7928,1687-0247},
   MRCLASS = {14C35 (14M15 19L47 20G05 22E46)},
  MRNUMBER = {4465997},
MRREVIEWER = {Changlong\ Zhong},
       DOI = {10.1093/imrn/rnab049},
       URL = {https://doi.org/10.1093/imrn/rnab049},
}

@article{knutson:noncomplex,
  title={A {S}chubert calculus recurrence from the noncomplex {W}-action on {G/B}},
  author={Knutson, Allen},
  journal={arXiv preprint math/0306304},
  year={2003}
}

@article{harada.landweber.sjamaar:divided,
  title={Divided differences and the {W}eyl character formula in equivariant {K}-theory},
  author={Harada, Megumi and Landweber, Gregory D and Sjamaar, Reyer},
  journal={Mathematical Research Letters},
  volume={17},
  number={3},
  pages={507--527},
  year={2010},
  publisher={International Press of Boston},
}

@article {postnikov:affine,
    AUTHOR = {Postnikov, Alexander},
     TITLE = {Affine approach to quantum {S}chubert calculus},
   JOURNAL = {Duke Math. J.},
  FJOURNAL = {Duke Mathematical Journal},
    VOLUME = {128},
      YEAR = {2005},
    NUMBER = {3},
     PAGES = {473--509},
      ISSN = {0012-7094,1547-7398},
   MRCLASS = {05E05 (14M15 14N15 14N35)},
  MRNUMBER = {2145741},
MRREVIEWER = {Harry\ Tamvakis},
       DOI = {10.1215/S0012-7094-04-12832-5},
       URL = {https://doi.org/10.1215/S0012-7094-04-12832-5},
}

@incollection {okounkov:lectures,
    AUTHOR = {Okounkov, Andrei},
     TITLE = {Lectures on {K}-theoretic computations in enumerative
              geometry},
 BOOKTITLE = {Geometry of moduli spaces and representation theory},
    SERIES = {IAS/Park City Math. Ser.},
    VOLUME = {24},
     PAGES = {251--380},
 PUBLISHER = {Amer. Math. Soc., Providence, RI},
      YEAR = {2017},
      ISBN = {978-1-4704-3574-5},
   MRCLASS = {14N35 (14C35 19M05)},
  MRNUMBER = {3752463},
MRREVIEWER = {Emily\ Clader},
       DOI = {10.1090/pcms/024},
       URL = {https://doi.org/10.1090/pcms/024},
}

@article {KPSZ:quivers,
    AUTHOR = {Koroteev, Peter and Pushkar, Petr P. and Smirnov, Andrey V.
              and Zeitlin, Anton M.},
     TITLE = {Quantum {K}-theory of quiver varieties and many-body systems},
   JOURNAL = {Selecta Math. (N.S.)},
  FJOURNAL = {Selecta Mathematica. New Series},
    VOLUME = {27},
      YEAR = {2021},
    NUMBER = {5},
     PAGES = {Paper No. 87, 40},
      ISSN = {1022-1824,1420-9020},
   MRCLASS = {19L47 (14C35 16G20)},
  MRNUMBER = {4308933},
MRREVIEWER = {\Dbar\cftil o{} Ng\d oc Di\cfudot ep},
       DOI = {10.1007/s00029-021-00698-3},
       URL = {https://doi.org/10.1007/s00029-021-00698-3},
}

@article{summers:qideal,
  author  = {Summers, Kevin},
  title   = {A Dual Basis for the Equivariant Quantum
             {K}-Theory of Cominuscule Varieties},
  journal = {Journal of Algebra},
  volume  = {710},
  pages   = {130--163},
  year    = {2027},
  doi     = {10.1016/j.jalgebra.2026.05.035},
  eprint  = {2407.02703},
  archivePrefix = {arXiv}
}

@article {mcnamara:factorial,
    AUTHOR = {McNamara, Peter J.},
     TITLE = {Factorial {G}rothendieck polynomials},
   JOURNAL = {Electron. J. Combin.},
  FJOURNAL = {Electronic Journal of Combinatorics},
    VOLUME = {13},
      YEAR = {2006},
    NUMBER = {1},
     PAGES = {Research Paper 71, 40},
      ISSN = {1077-8926},
   MRCLASS = {05E05},
  MRNUMBER = {2240776},
MRREVIEWER = {Pavlo\ Pylyavskyy},
       DOI = {10.37236/1097},
       URL = {https://doi.org/10.37236/1097},
}

@article {gorbounov.korff.stroppel,
    AUTHOR = {Gorbunov, V. G. and Korff, Ch. and Stroppel, C.},
     TITLE = {Yang-{B}axter algebras, convolution algebras, and
              {G}rassmannians},
   JOURNAL = {Uspekhi Mat. Nauk},
  FJOURNAL = {Uspekhi Matematicheskikh Nauk},
    VOLUME = {75},
      YEAR = {2020},
    NUMBER = {5(455)},
     PAGES = {3--58},
      ISSN = {0042-1316,2305-2872},
   MRCLASS = {16T25 (14M15 16G20 81R12)},
  MRNUMBER = {4154847},
       DOI = {10.4213/rm9959},
       URL = {https://doi.org/10.4213/rm9959},
}

@article{gorbounov.korff:unpublished,
author = {Gorbounov, Vassily and Korff, Christian},
title = {Equivariant quantum cohomology and Yang-Baxter algebras},
journal={available on arXiv:1402.2907v2},
year={2014},
}

@article{BCMP:qkchev,
    AUTHOR = {Buch, Anders S. and Chaput, Pierre-Emmanuel and Mihalcea,
              Leonardo C. and Perrin, Nicolas},
     TITLE = {A {C}hevalley formula for the equivariant quantum {$K$}-theory of cominuscule varieties},
   JOURNAL = {Algebr. Geom.},
  FJOURNAL = {Algebraic Geometry},
    VOLUME = {5},
      YEAR = {2018},
    NUMBER = {5},
     PAGES = {568--595},
      ISSN = {2313-1691,2214-2584},
   MRCLASS = {14N35 (14C35 14M15 14N15 19L47)},
  MRNUMBER = {3847206},
MRREVIEWER = {Feng\ Qu},
       DOI = {10.14231/ag-2018-015},
       URL = {https://doi.org/10.14231/ag-2018-015},
}

@article {anderson.griffeth.miller:Kpos,
    AUTHOR = {Anderson, Dave and Griffeth, Stephen and Miller, Ezra},
     TITLE = {Positivity and {K}leiman transversality in equivariant
              {$K$}-theory of homogeneous spaces},
   JOURNAL = {J. Eur. Math. Soc. (JEMS)},
  FJOURNAL = {Journal of the European Mathematical Society (JEMS)},
    VOLUME = {13},
      YEAR = {2011},
    NUMBER = {1},
     PAGES = {57--84},
      ISSN = {1435-9855,1435-9863},
   MRCLASS = {14C35 (14C17 14M15 19L47)},
  MRNUMBER = {2735076},
MRREVIEWER = {Harry\ Tamvakis},
       DOI = {10.4171/JEMS/244},
       URL = {https://doi-org.ezproxy.lib.vt.edu/10.4171/JEMS/244},
}

@article {knutson.lam.speyer:projrich,
    AUTHOR = {Knutson, Allen and Lam, Thomas and Speyer, David E.},
     TITLE = {Projections of {R}ichardson varieties},
   JOURNAL = {J. Reine Angew. Math.},
  FJOURNAL = {Journal f\"ur die Reine und Angewandte Mathematik. [Crelle's
              Journal]},
    VOLUME = {687},
      YEAR = {2014},
     PAGES = {133--157},
      ISSN = {0075-4102,1435-5345},
   MRCLASS = {05E18 (14L35 14M15)},
  MRNUMBER = {3176610},
MRREVIEWER = {Anthony\ Henderson},
       DOI = {10.1515/crelle-2012-0045},
       URL = {https://doi-org.ezproxy.lib.vt.edu/10.1515/crelle-2012-0045},
}

@article {knutson.lam.speyer:juggling,
    AUTHOR = {Knutson, Allen and Lam, Thomas and Speyer, David E.},
     TITLE = {Positroid varieties: juggling and geometry},
   JOURNAL = {Compos. Math.},
  FJOURNAL = {Compositio Mathematica},
    VOLUME = {149},
      YEAR = {2013},
    NUMBER = {10},
     PAGES = {1710--1752},
      ISSN = {0010-437X,1570-5846},
   MRCLASS = {14M15 (05B35 05E05)},
  MRNUMBER = {3123307},
MRREVIEWER = {Jacopo\ Gandini},
       DOI = {10.1112/S0010437X13007240},
       URL = {https://doi-org.ezproxy.lib.vt.edu/10.1112/S0010437X13007240},
}

@book{kock2004frobenius,
  title={Frobenius algebras and 2-d topological quantum field theories},
  author={Kock, Joachim},
  number={59},
  year={2004},
  publisher={Cambridge University Press}
}

@article{abrams1996two,
  title={Two-dimensional topological quantum field theories and {F}robenius algebras},
  author={Abrams, Lowell},
  journal={Journal of Knot theory and its ramifications},
  volume={5},
  number={05},
  pages={569--587},
  year={1996},
  publisher={World Scientific}
}

@article{lenart2006grothendieck,
  title={Grothendieck polynomials via permutation patterns and chains in the {B}ruhat order},
  author={Lenart, Cristian and Robinson, Shawn and Sottile, Frank},
  journal={American Journal of Mathematics},
  volume={128},
  number={4},
  pages={805--848},
  year={2006},
  publisher={Johns Hopkins University Press}
}

@article{yeliussizov2019symmetric,
  title={Symmetric {G}rothendieck polynomials, skew {C}auchy identities, and dual filtered {Y}oung graphs},
  author={Yeliussizov, Damir},
  journal={Journal of Combinatorial Theory, Series A},
  volume={161},
  pages={453--485},
  year={2019},
  publisher={Elsevier}
}

@article{korff2014quantum,
  title={Quantum cohomology via vicious and osculating walkers},
  author={Korff, Christian},
  journal={Letters in Mathematical Physics},
  volume={104},
  number={7},
  pages={771--810},
  year={2014},
  publisher={Springer}
}

@article{gorbounov.korff.mihalcea:YB,
author = {Gorbounov, V. and Korff, C. and Mihalcea, L.~C.},
title = {Quantum {K}-theory of {G}rassmannians from a {Y}ang-{B}axter algebra},
journal={submitted, preprint available at arXiv:2503.08602},
year={2025},
}

@article{gessel1997cylindric,
  title={Cylindric partitions},
  author={Gessel, Ira and Krattenthaler, Christian},
  journal={Transactions of the American Mathematical Society},
  volume={349},
  number={2},
  pages={429--479},
  year={1997}
}

@article {fulton.woodward:quantum,
    AUTHOR = {Fulton, W. and Woodward, C.},
     TITLE = {On the quantum product of {S}chubert classes},
   JOURNAL = {J. Algebraic Geom.},
  FJOURNAL = {Journal of Algebraic Geometry},
    VOLUME = {13},
      YEAR = {2004},
    NUMBER = {4},
     PAGES = {641--661},
      ISSN = {1056-3911,1534-7486},
   MRCLASS = {14N35 (14M15 14N15)},
  MRNUMBER = {2072765},
MRREVIEWER = {Harry\ Tamvakis},
       DOI = {10.1090/S1056-3911-04-00365-0},
       URL = {https://doi-org.ezproxy.lib.vt.edu/10.1090/S1056-3911-04-00365-0},
}

@article {nekrasov.shatashvili:supersymmetric,
    AUTHOR = {Nekrasov, Nikita A. and Shatashvili, Samson L.},
     TITLE = {Supersymmetric vacua and {B}ethe ansatz},
   JOURNAL = {Nuclear Phys. B Proc. Suppl.},
  FJOURNAL = {Nuclear Physics B. Proceedings Supplement},
    VOLUME = {192/193},
      YEAR = {2009},
     PAGES = {91--112},
      ISSN = {0920-5632},
   MRCLASS = {81T60 (81R12)},
  MRNUMBER = {2570974},
       DOI = {10.1016/j.nuclphysbps.2009.07.047},
       URL = {https://doi.org/10.1016/j.nuclphysbps.2009.07.047},
}

@article{Buch2002,
  author  = {Anders Skovsted Buch},
  title   = {A {L}ittlewood--{R}ichardson Rule for the {K}-Theory of {G}rassmannians},
  journal = {Acta Mathematica},
  volume  = {189},
  number  = {1},
  pages   = {37--78},
  year    = {2002},
  doi     = {10.1007/BF02392644},
  eprint  = {math/0004137},
  archivePrefix = {arXiv}
}

@article{Nielsen1974,
  author  = {Nielsen, H. Andreas},
  title   = {Diagonalizably linearized coherent sheaves},
  journal = {Bulletin de la Soci{\'e}t{\'e} Math{\'e}matique de France},
  volume  = {102},
  pages   = {85--97},
  year    = {1974},
  doi     = {10.24033/bsmf.1771},
  mrnumber = {51 #3174},
  zbl     = {0293.14006}
}

@article{ikeda2013k,
  title={K-theoretic analogues of factorial {S}chur {P}-and {Q}-functions},
  author={Ikeda, Takeshi and Naruse, Hiroshi},
  journal={Advances in Mathematics},
  volume={243},
  pages={22--66},
  year={2013},
  publisher={Elsevier}
}

@article{lieb1967residual,
  title={Residual entropy of square ice},
  author={Lieb, Elliott H},
  journal={Physical Review},
  volume={162},
  number={1},
  pages={162},
  year={1967},
  publisher={APS}
}

@article{sutherland1967exact,
  title={Exact solution of a two-dimensional model for hydrogen-bonded crystals},
  author={Sutherland, Bill},
  journal={Physical Review Letters},
  volume={19},
  number={3},
  pages={103},
  year={1967},
  publisher={APS}
}

@article{fisher1984walks,
  title={Walks, walls, wetting, and melting},
  author={Fisher, Michael E},
  journal={Journal of Statistical Physics},
  volume={34},
  number={5},
  pages={667--729},
  year={1984},
  publisher={Springer}
}

@article{brak1999bethe,
  title={From the {B}ethe ansatz to the {Gessel-Viennot} theorem},
  author={Brak, R and Essam, JW and Owczarek, AL},
  journal={Annals of Combinatorics},
  volume={3},
  number={2},
  pages={251--263},
  year={1999},
  publisher={Springer}
}

@inproceedings{bousquet2006three,
  title={Three osculating walkers},
  author={Bousquet-M{\'e}lou, Mireille},
  booktitle={Journal of Physics: Conference Series},
  volume={42},
  number={1},
  pages={35--46},
  year={2006}
}

@article{Huang1996,
title = {Interacting dimers on the honeycomb lattice: an exact solution of the five-vertex model},
journal = {Physica A: Statistical Mechanics and its Applications},
volume = {228},
number = {1},
pages = {1-32},
year = {1996},
issn = {0378-4371},
doi = {https://doi.org/10.1016/S0378-4371(96)00057-X},
url = {https://www.sciencedirect.com/science/article/pii/S037843719600057X},
author = {H.Y. Huang and F.Y. Wu and H. Kunz and D. Kim}
}

@article{lascoux1982structurehopf,
  author   = {Alain Lascoux and Marcel-Paul Sch{\"u}tzenberger},
  title    = {{Structure de Hopf de l'anneau de cohomologie et de l'anneau de Grothendieck d'une vari{\'e}t{\'e} de drapeaux}},
  journal  = {C. R. Acad. Sci. Paris S{\'e}r. I Math.},
  volume   = {295},
  number   = {11},
  pages    = {629--633},
  year     = {1982}
}

@misc{zinnjustin2009sixvertex,
      title={Six-Vertex, Loop and Tiling models: Integrability and Combinatorics}, 
      author={P. Zinn-Justin},
      year={2009},
      eprint={0901.0665},
      archivePrefix={arXiv},
      primaryClass={math-ph},
      url={https://arxiv.org/abs/0901.0665}, 
}

@article{brubaker2011schur,
  author  = {Brubaker, Ben and Bump, Daniel and
             Friedberg, Solomon},
  title   = {{Schur} Polynomials and the {Yang--Baxter} Equation},
  journal = {Communications in Mathematical Physics},
  volume  = {308},
  number  = {2},
  pages   = {281--301},
  year    = {2011},
  doi     = {10.1007/s00220-011-1345-3}
}

@article {bump:mcnamara:nakasuji,
    AUTHOR = {Bump, Daniel and McNamara, Peter J. and Nakasuji, Maki},
     TITLE = {Factorial {S}chur functions and the {Y}ang-{B}axter equation},
   JOURNAL = {Comment. Math. Univ. St. Pauli},
  FJOURNAL = {Commentarii Mathematici Universitatis Sancti Pauli},
    VOLUME = {63},
      YEAR = {2014},
    NUMBER = {1-2},
     PAGES = {23--45},
      ISSN = {0010-258X},
   MRCLASS = {05E05},
  MRNUMBER = {3328423},
MRREVIEWER = {Eric S. Egge},
}

@article{brubaker2020vertex,
  author  = {Brubaker, Ben and Buciumas, Valentin and
             Bump, Daniel and Gustafsson, Henrik P. A.},
  title   = {Vertex Operators, Solvable Lattice Models and
             Metaplectic {Whittaker} Functions},
  journal = {Communications in Mathematical Physics},
  volume  = {380},
  number  = {2},
  pages   = {535--579},
  year    = {2020},
  doi     = {10.1007/s00220-020-03842-w}
}

@article{brubaker2021colored,
  author  = {Brubaker, Ben and Buciumas, Valentin and
             Bump, Daniel and Gustafsson, Henrik P. A.},
  title   = {Colored Five-Vertex Models and {D}emazure Atoms},
  journal = {Journal of Combinatorial Theory, Series A},
  volume  = {178},
  pages   = {105354},
  year    = {2021},
  doi     = {10.1016/j.jcta.2020.105354}
}

@article{korff21,
  title={Cylindric {H}ecke Characters and {G}romov--{W}itten Invariants via the Asymmetric Six-Vertex Model},
  author={Korff, Christian},
  journal={Communications in Mathematical Physics},
  volume={381},
  number={2},
  pages={591--640},
  year={2021},
  publisher={Springer Berlin Heidelberg},
  eprint={https://arxiv.org/abs/1906.02565},
}

@article{buciumas2022double,
  title={Double {G}rothendieck polynomials and colored lattice models},
  author={Buciumas, Valentin and Scrimshaw, Travis},
  journal={International Mathematics Research Notices},
  volume={2022},
  number={10},
  pages={7231--7258},
  year={2022},
  publisher={Oxford University Press}
}

@article{corteel2022vertex,
  title={A vertex model for {LLT} polynomials},
  author={Corteel, Sylvie and Gitlin, Andrew and Keating, David and Meza, Jeremy},
  journal={International Mathematics Research Notices},
  volume={2022},
  number={20},
  pages={15869--15931},
  year={2022},
  publisher={Oxford University Press}
}

@article{brubaker2023frozen,
  title={Frozen pipes: lattice models for Grothendieck polynomials},
  author={Brubaker, Ben and Frechette, Claire and Hardt, Andrew and Tibor, Emily and Weber, Katherine},
  journal={Algebraic Combinatorics},
  volume={6},
  number={3},
  pages={789--833},
  year={2023}
}

@article{aggarwal2023free,
  title={Free fermion six vertex model: symmetric functions and random domino tilings},
  author={Aggarwal, Amol and Borodin, Alexei and Petrov, Leonid and Wheeler, Michael},
  journal={Selecta Mathematica},
  volume={29},
  number={3},
  pages={36},
  year={2023},
  publisher={Springer}
}

@article{naprienko2024free,
  title={Free fermionic Schur functions},
  author={Naprienko, Slava},
  journal={Advances in Mathematics},
  volume={436},
  pages={109413},
  year={2024},
  publisher={Elsevier}
}

@article{brubaker2025lattice,
  title={Lattice Models for Double {W}hittaker Polynomials and Motivic {C}hern Classes},
  author={Brubaker, Ben and Bump, Daniel and Hardt, Andrew and Spink, Hunter},
  journal={arXiv preprint arXiv:2509.17312},
  year={2025}
}

@book{FultonYoungTableaux,
  author    = {Fulton, William},
  title     = {Young {T}ableaux: With Applications to Representation Theory and Geometry},
  series    = {London Mathematical Society Student Texts},
  volume    = {35},
  publisher = {Cambridge University Press},
  address   = {Cambridge},
  year      = {1997}
}

@article{Graham2001,
  author  = {Graham, William},
  title   = {Positivity in Equivariant {S}chubert Calculus},
  journal = {Duke Mathematical Journal},
  volume  = {109},
  number  = {3},
  pages   = {599--614},
  year    = {2001},
  doi     = {10.1215/S0012-7094-01-10935-6}
}

@article{Brion2002,
  author  = {Brion, Michel},
  title   = {Positivity in the {G}rothendieck Group of Complex Flag Varieties},
  journal = {Journal of Algebra},
  volume  = {258},
  number  = {1},
  pages   = {137--159},
  year    = {2002},
  doi     = {10.1016/S0021-8693(02)00505-7}
}

@article{Buch2015,
  author  = {Buch, Anders Skovsted},
  title   = {Mutations of Puzzles and Equivariant Cohomology of Two-Step Flag Varieties},
  journal = {Annals of Mathematics},
  volume  = {182},
  number  = {1},
  pages   = {173--220},
  year    = {2015},
  doi     = {10.4007/annals.2015.182.1.4}
}

@article{BKPT2016,
  author  = {Buch, Anders Skovsted and Kresch, Andrew and Purbhoo, Kevin and Tamvakis, Harry},
  title   = {The Puzzle Conjecture for the Cohomology of Two-Step Flag Manifolds},
  journal = {Journal of Algebraic Combinatorics},
  volume  = {44},
  number  = {4},
  pages   = {973--1007},
  year    = {2016},
  doi     = {10.1007/s10801-016-0697-3}
}

@article{KnutsonZinnJustin2026,
  author  = {Knutson, Allen and Zinn-Justin, Paul},
  title   = {Schubert Puzzles and Integrability {I}: Invariant Trilinear Forms},
  journal = {Communications of the American Mathematical Society},
  volume  = {6},
  pages   = {1--67},
  year    = {2026},
  doi     = {10.1090/cams/56},
  eprint  = {1706.10019},
  archivePrefix = {arXiv}
}

@article{Mihalcea2006,
  author  = {Mihalcea, Leonardo C.},
  title   = {Equivariant Quantum {S}chubert Calculus},
  journal = {Advances in Mathematics},
  volume  = {203},
  number  = {1},
  pages   = {1--33},
  year    = {2006},
  doi     = {10.1016/j.aim.2005.04.002}
}

@article{FultonLascoux1994,
  author  = {Fulton, William and Lascoux, Alain},
  title   = {A {P}ieri Formula in the {G}rothendieck Ring of a Flag Bundle},
  journal = {Duke Mathematical Journal},
  volume  = {76},
  number  = {3},
  pages   = {711--729},
  year    = {1994},
  doi     = {10.1215/S0012-7094-94-07627-8}
}

@misc{Oet21,
  author        = {Oetjen, David},
  title         = {Factorial {Schur} and {Grothendieck}
                   Polynomials from {Bott--Samelson}
                   Varieties},
  year          = {2021},
  eprint        = {2109.10483},
  archivePrefix = {arXiv},
  primaryClass  = {math.AG}
}
\end{document}